%% file: SWW_Arxiv_1.tex
\documentclass[a4paper,11pt]{amsart}
\usepackage{tikz,tikz-cd,amscd,amsmath,amsthm,amssymb,braket,mathrsfs,graphicx,verbatim,mathtools}
\usepackage{pgfplots}
\usepackage{hyperref}
\hypersetup{colorlinks=true, linkcolor=black,citecolor=black}
\usepackage{setspace}

\usepackage[margin=3cm]{geometry}

\usetikzlibrary{matrix,arrows,decorations.pathmorphing}
\usepackage[all]{xy}

\makeatletter
\newtheorem*{rep@theorem}{\rep@title}
\newcommand{\newreptheorem}[2]{%
\newenvironment{rep#1}[1]{%
 \def\rep@title{#2 \ref{##1}}%
 \begin{rep@theorem}}%
 {\end{rep@theorem}}}
\makeatother

\let\oldtocsection=\tocsection
 
\let\oldtocsubsection=\tocsubsection
 
\let\oldtocsubsubsection=\tocsubsubsection
 
\renewcommand{\tocsection}[2]{\hspace{0em}\oldtocsection{#1}{#2}\vspace*{-1.5mm}}
\renewcommand{\tocsubsection}[2]{\hspace{1em}\oldtocsubsection{#1}{#2}\vspace*{-1.5mm}}
\renewcommand{\tocsubsubsection}[2]{\hspace{2em}\oldtocsubsubsection{#1}{#2}\vspace*{-1.5mm}}

\newreptheorem{theorem}{Theorem}
\newtheorem{theorem}{Theorem}[section]
\theoremstyle{definition}\newtheorem*{definition}{Definition}
\theoremstyle{definition}\newtheorem{proposition}[theorem]{Proposition}
\theoremstyle{definition}\newtheorem{lemma}[theorem]{Lemma}
\theoremstyle{definition}\newtheorem{corollary}[theorem]{Corollary}
\theoremstyle{definition}
\theoremstyle{remark}\newtheorem{remark}[theorem]{Remark}
\usepackage{xcolor}
\definecolor{steel}{HTML}{396a93}

\newcommand{\IM}{\text{\normalfont Imm}^1}
\newcommand{\SL}{\mathcal{L}}
\newcommand{\SC}{\mathcal{C}}
\newcommand{\SI}{\mathcal{I}}
\newcommand{\SB}{\mathcal{B}}
\newcommand{\dst}{\text{\normalfont d}_{L^2(ds)}}
\newcommand{\vn}[1]{\lVert{#1}\rVert}

\def\D{\mathcal{D}}

\def\R{\mathbb{R}}
\def\N{\mathbb{N}}

\def\C{\mathcal{C}\!h}

\def\SQ{\mathcal{Q}}
\def\mcs{\mathcal{S}}

\def\h1{{H^1}}
\def\l2{{L^2}}
\renewcommand{\S}{\mathbb{S}}
\newcommand{\norm}[1]{\vn{#1}}
\newcommand{\enorm}[1]{{\left | {#1}\right |}}
\newcommand{\ip}[1]{{\left \langle {#1} \right \rangle}}

\DeclareMathOperator{\imm}{\text{\normalfont Imm}}

\DeclareMathOperator{\grad}{grad}
\DeclareMathOperator{\diff}{Diff}
\def\L{\mathcal L}
  
\DeclareMathOperator{\sign}{sgn}

\makeatletter
\@namedef{subjclassname@2020}{%
  \textup{2020} Mathematics Subject Classification}
\makeatother

\begin{document}

\title{On the $H^1(ds)$-gradient flow for the length functional}

\author[P. Schrader]{Philip Schrader}
\address{JSPS International Research Fellow \\ Mathematical Institute\\ Tohoku University\\
Sendai 980-8578\\
Japan}
\email{philschrad@gmail.com}

\author[G. Wheeler]{Glen Wheeler}
\address{Institute for Mathematics and its Applications, School of Mathematics and Applied Statistics\\
University of Wollongong\\
Northfields Ave, Wollongong, NSW $2500$\\
Australia}
\email{glenw@uow.edu.au}

\author[V.-M. Wheeler]{Valentina-Mira Wheeler}
\address{Institute for Mathematics and its Applications, School of Mathematics and Applied Statistics\\
University of Wollongong\\
Northfields Ave, Wollongong, NSW $2500$\\
Australia}
\email{vwheeler@uow.edu.au}

\thanks{This research was supported in part by
Discovery Project DP180100431 and DECRA DE190100379
of the Australian Research Council, and an International Postdoctoral Research
Fellowship of the Japan Society for the Promotion of Science.}

\maketitle

\begin{abstract}
In this article we consider the length functional defined on the space of
immersed planar curves.
The $L^2(ds)$ Riemannian metric gives rise to the curve shortening flow as the
gradient flow of the length functional.
Motivated by the vanishing of the $L^2(ds)$ Riemannian distance, we consider
the gradient flow of the length functional with respect to the
$H^1(ds)$-metric.
Circles with radius $r_0$ shrink with $r(t) = \sqrt{W(e^{c-2t})}$ under the
flow, where $W$ is the Lambert $W$ function and $c = r_0^2 + \log r_0^2$.
We conduct a thorough study of this flow, giving existence of eternal solutions
and convergence for general initial data, preservation of regularity in various
spaces, qualitative properties of the flow after an appropriate rescaling, and
numerical simulations.

\end{abstract}

\tableofcontents 

\section{Introduction}

Consider the length functional: 
\begin{equation}\label{eq:length}
\L(\gamma):=\int_{\S^1} \enorm{\gamma'}du \end{equation}
defined on $\IM$, the space of (once) differentiable curves $\gamma:\S^1\to \R^2$ with
$|\gamma'|\ne0$.
The definition of a gradient of $\SL$
requires a notion of direction on $\IM$,
that is an inner product or more generally a Riemannian metric ${\langle \cdot,
\cdot \rangle}$.
The gradient is then characterised by $ d\SL={\langle\grad \SL, \cdot \rangle}$. 

To calculate the (Gateaux) derivative $d\L $ take a variation $\gamma
:(-\varepsilon,\varepsilon)\times \S^1 \to \R^2$, ${\left . \partial_\varepsilon
\gamma \right|}_{\varepsilon=0}=V$ and calculate
\begin{equation}
\begin{split}
d\L_\gamma V = {\left .  \int_{\S^1}\frac{ \langle\partial_\varepsilon \partial_u \gamma , \partial_u\gamma\rangle  }{|\partial_u \gamma|} du \right |}_{\varepsilon=0}
 \label{eq:dlength}
=-\int \langle T_s,V\rangle\, ds 
=-\int k\ip{N,V}\,ds
\end{split}
\end{equation}
 Here the inner product is the Euclidean one, $u$ is the given parameter along $\gamma=(x,y)$,   $s$ is the Euclidean arc-length parameter, $T=(x_s,y_s)$ is the unit
tangent vector, $k$ the curvature scalar, and $N
= (-y_s,x_s)$ is the normal vector. 

As for the inner product or Riemannian metric in $\IM$ we might choose either of 
\begin{align*}
\ip{v,w}_{L^2} :=\int_{ \S^1} \ip{v,w} du\,,\quad\text{or}\quad
	{\langle v,w\rangle}_{L^2(ds)} := \int_{\S^1}\ip{v,w} \, ds 
\end{align*}
for $v,w$ vector fields along $\gamma$. The former is simpler, but
from the point of view
of geometric analysis (and in particular geometric flows) the latter is preferable
because it is invariant under reparametrisation of $\gamma$, and this invariance
carries through to the corresponding gradient flow (see Section
\ref{symmetries}).  Note that
the $L^2(ds)$ product is in fact a Riemannian metric because it depends on the
base point $\gamma$ through the measure $ds$.

\subsection{The gradient flow for length in $(\IM,L^2(ds))$}
Indeed \eqref{eq:dlength} shows that the gradient flow of length in the $L^2(ds)$ metric is the famous curve shortening flow proposed by Gage-Hamilton \cite{GH}:
\begin{equation}\label{eq:csf}
	X_t = X_{ss} = \kappa = kN 
\end{equation}
where $X:\S\times(0,T)\rightarrow\R^2$ is a one-parameter family of immersed regular closed curves, $X(u,t) = (x(u,t),y(u,t))$ and $s,T,k,N$ are as above.

The curve shortening flow moves each point along a curve in the direction of
the curvature vector at that point.
Concerning local and global behaviour of the flow, we have:

\begin{theorem}[Angenent \cite{A}, Grayson \cite{Gray}, Gage-Hamilton \cite{GH}, Ecker-Huisken \cite{EH}]
\label{TMcsf}
Consider a locally Lipschitz embedded curve $X_0$.
There exists a curve shortening flow $X:\S\times(0,T)\rightarrow\R^2$ such that $X(\cdot,t)\searrow X_0$ in the $C^{1/2}$-topology.
The maximal time of smooth existence for the flow is finite, and as $t\nearrow T$, $X(\cdot,t)$ shrinks to a point $\{p\}$.
The normalised flow with length or area fixed exists for all time.
It becomes eventually convex, and converges exponentially fast to a standard round circle in the smooth topology.
\end{theorem}

\begin{remark}
In the theorem above, we make the following attributions.
Angenent \cite{A} showed that the curve shortening flow exists with locally
Lipschitz data where convergence as $t\searrow0$ is in the continuous topology.
Ecker-Huisken's interior estimates in \cite{EH} extend this to the $C^{1/2}$-topology.
Gage-Hamilton \cite{GH} showed that a convex curve contracts to a round point,
whereas Grayson \cite{Gray} proved that any embedded curve becomes eventually
convex.
There are a number of ways that this can be proved; for instance we also
mention Huisken's distance comparison \cite{Hcsf} and the novel optimal
curvature estimate method in \cite{ABryan}.
\end{remark}

The curve shortening flow has been extensively studied and found many
applications. We refer the interested reader to the recent book
\cite{Andrews:2020aa}.

\subsection{Vanishing Riemannian distance in $(\IM,L^2(ds))$}
Every Riemannian metric induces a distance function defined as the infimum of
lengths of paths joining two points. For finite dimensional manifolds the
resulting path-metric space has the same topology as the manifold, but for
infinite-dimensional manifolds it is possible that the path-metric space
topology is weaker (so-called \emph{weak Riemannian metrics}). Furthermore, as
famously demonstrated by Michor-Mumford in \cite{MR2201275}, it is possible
that the Riemannian distance is actually trivial. 

The example given by Michor-Mumford is the space $(\IM,L^2(ds))$ quotient by the diffeomorphism
group\footnote{$\text{Diff}(\S^1)$ is the regular Lie group
of all diffeomorphisms $\phi:\S^1\rightarrow\S^1$ with connected components
$\text{Diff}^+(\S^1)$, $\text{Diff}^-(\S^1)$ given by orientation preserving
and orientation reversing diffeomorphisms respectively.} $\text{Diff}(\S^1)$
 of $\S^1$.
We call this space $\SQ = \text{Imm}^1 / \text{Diff}(\S^1)$.
While $\text{Imm}^1$ is an open subset of $C^1(\S^1,\R^2)$ and so $(\text{Imm}^1,L^2(ds))$ is a Riemannian manifold, the action of $\text{Diff}(\S^1)$ is not free (see
\cite[Sections 2.4 and 2.5]{MR2201275}), and so the quotient $\SQ$ is not a
manifold but is an orbifold.

\begin{reptheorem}{TMMM}[Michor-Mumford \cite{MR2201275}]
The Riemannian distance in $(\SQ, L^2(ds))$ is trivial.
\end{reptheorem}
This surprising fact is shown by an explicit construction in
\cite{MR2201275} of a path between orbits with arbitrarily small
$L^2(ds)$-length, which for the benefit of the reader we briefly recall in
Section \ref{SNawesome}.
A natural question arising from Michor-Mumford's work is if the induced metric
topology on the Riemannian manifold $(\IM,L^2(ds))$ is also trivial.
This was confirmed in \cite{Bauer:2012aa} as a special case of a more general result. 

\begin{reptheorem}{TMawesome}
The Riemannian distance in $(\IM, L^2(ds))$ is trivial.
\end{reptheorem}

Here we give a different proof of Theorem \ref{TMawesome} using a detour through small curves. The setup and proof is given in detail in Section \ref{SNawesome}.

We can see from \eqref{eq:dlength} that the curve shortening flow
\eqref{eq:csf} is indeed the $L^2(ds)$-gradient flow of the length functional
in $\IM$, not the quotient $\SQ$.
Theorem \ref{TMawesome} yields that the underlying metric space that the curve
shortening flow is defined upon is \emph{trivial}, and therefore
this background metric space structure is useless in the analysis
of the flow.

While it could conceivably be true that the triviality of the Riemannian metric topology
on $\IM, L^2(ds)$  is important for the validity of Theorem
\ref{TMcsf} and the other nice properties that the curve shortening flow
enjoys, one naturally wonders if this is in fact the case: What do gradient
flows of length look like on $\IM$, with other choices of Riemannian metric?

\subsection{The gradient flow for length in $(\IM,H^1(ds))$}
We wish to choose a metric that (a) yields a non-trivial Riemannian distance; and (b) produces a non-trivial gradient flow.
One way of doing this (similar to that described by Michor-Mumford \cite{MR2201275}) is to view the $L^2(ds)$ Riemannian metric as an element on the Sobolev scale of metrics (as the $H^0(ds)$ metric).
The next most simple choice is therefore the $H^1(ds)$ metric:
\begin{equation}\label{eq:h1dsip}
\ip{v,w}_{H^1(ds)}:=\ip{v,w}_{L^2(ds)}+\ip{v_s,w_s}_{L^2(ds)}
\end{equation}
Note that we have set the parameter $A$ from \cite[Section 3.2, Equation (5)]{MR2201275} to $1$ and we are considering the full space, not the quotient.
In contrast to the $L^2(ds)$ case, the $H^1(ds)$ distance is non-trivial \cite{MR2201275}.

\begin{remark}
There is an expanding literature on the multitude of
alternative metrics proposed for quantitative comparison of shapes in imaging
applications (see for example \cite{bauer2014constructing,bauer2015metrics,epifanio2020new,kurtek2011elastic,mennucci2019designing,MR2201275,nardi2016geodesics,shah2013h2,Sundaramoorthi2007,tumpach2017quotient,younes1998computable,younes2018hybrid,younes2008metric}).
It might be interesting to compare the dynamical properties of the gradient
flows of length on $\IM$ with respect to other Riemannian metrics.  We note also that the study of Sobolev type gradients is far from new. We mention the comprehensive book on the topic by Neuberger \cite{Neuberger:2010aa}, and the flow studied in \cite{Sundaramoorthi2007} for applications to active contours is closely related to the one we study here. A recurring theme seems to be better numerical stability for the Sobolev gradient compared to its $L^2$ counterpart. However, in this article we focus on analytical aspects. 
\end{remark}

The steepest descent $H^1(ds)$-gradient flow for length (called the
\emph{$H^1(ds)$ curve shortening flow}) on maps in $H^1(\S^1,\R^2)$ is a
one-parameter family of maps $X:\S^1\times I\rightarrow\R^2$ ($I$ an interval
containing zero) where for each $t$, $X(\cdot,t)\in H^1(\S^1,\R^2)$ and
\begin{equation}\label{h1flow}
	 \partial_t X(s,t)
	= -\Big(\text{grad}_{H^1(ds)}\L_{X(\cdot,t)}\Big)(s)
	= -X(s,t)-\int_0^\L X(\tilde s,t)G(s,\tilde s)d\tilde s
\end{equation}
where $G$ is given by
\begin{equation*}
G(s,\tilde s)=\frac{\cosh\left (|s-\tilde s|-\tfrac{\L}{2}\right)}{2\sinh(-\tfrac{\L}{2})} \quad \text{ for } 0\leq s,\tilde s \leq \L
\,.
\end{equation*}
Our derivation of this is contained in Section 4.1.

An instructive example of the flow's behaviour is exhibited by taking any standard round circle as initial data.
A circle will shrink self-similarly to a point under the flow, taking infinite time to do so.
The circle solutions can be extended uniquely and indefinitely in negative time as well, that is, they are eternal
solutions (see also Section 4.1). 

We set $\SC$ to be the space of constant maps.
While \eqref{h1flow} does not make sense on $\SC$, our first main result is that
everywhere else on $H^1(\S^1,\R^2)$ it does, and we are able to obtain eternal
solutions for any initial data $X_0\in H^1(\S^1,\R^2)\setminus \SC$.
This is Theorem \ref{TMglob}, which is the main result of Sections \ref{ssexistence}--\ref{ssglobal}:

\begin{reptheorem}{TMglob}
For each $X_0\in H^1(\S^1,\R^2)\setminus\SC$ there exists a unique eternal
$H^1(ds)$ curve shortening flow $X:\S^1\times\R\rightarrow\R^2$ in
$C^1(\R; H^1(\S^1,\R^2)\setminus\SC)$ such that $X(\cdot,0) = X_0$.
\end{reptheorem}

In Section \ref{ssconvergence} we study convergence for the flow, showing that the flow is asymptotic to a constant map in $\SC$.

\begin{reptheorem}{convergence}
Let $X$ be an $H^1(ds)$ curve shortening flow.
Then $X$ converges as $t\rightarrow\infty$ in $H^1$ to a constant map $X_\infty\in H^1(\S^1,\R^2)$.
\end{reptheorem}

Numerical simulations of the flow show fascinating qualitative behaviour for solutions.
Figures \ref{fig:squares} and \ref{fig:barbell} exhibit three important properties: first, that there
is no smoothing effect - it appears to be possible for corners to persist throughout the flow.
Second, that the evolution of a given initial curve is highly dependent upon
its size, to the extent that a simple rescaling dramatically alters the the amount of re-shaping along the flow. 
Third, the numerical simulations in Figures \ref{fig:squares} and \ref{fig:barbell} indicate that the flow does
not uniformly move curves closer to circles.
The scale-invariant isoperimetric ratio $\SI$ is plotted alongside the
evolutions and for an embedded barbell it is not monotone\footnote{This is also the case for the classical curve shortening flow, as explained in \cite{Gage:1983aa}.}.
We have given some comments on our numerical scheme and a link to the code in
Section \ref{ssnumerics}.
\begin{figure}[h] 
\includegraphics[width=10cm]{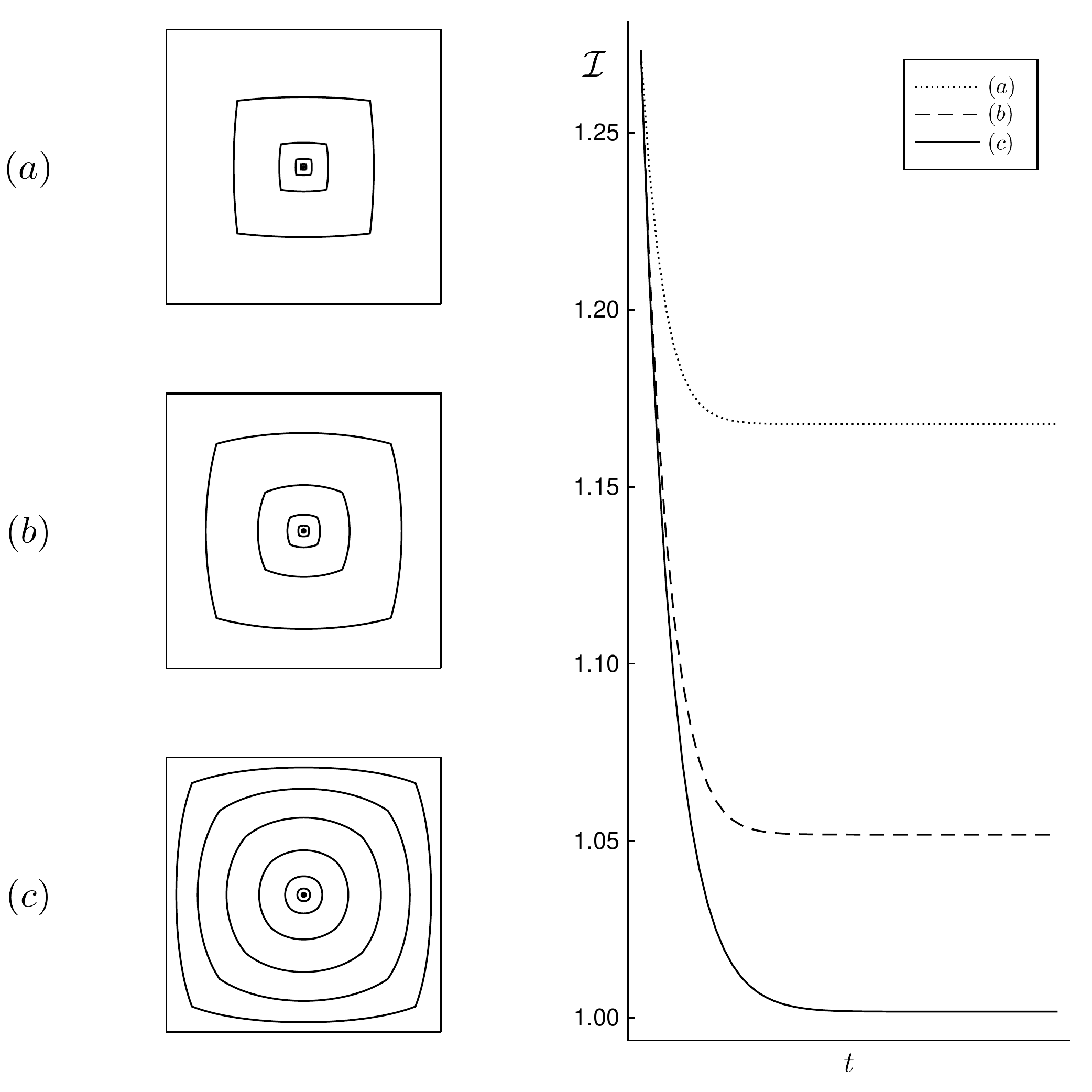}
\caption{ Left: (a),(b),(c) initial side lengths 1,2,4 respectively, time step $0.2$, every 5th out of 50 steps shown. Right: evolution of the isoperimetric ratio.}
\label{fig:squares}
\end{figure}

Despite the lack of a generic smoothing effect, what we might hope is that a generic \emph{preservation} effect holds.
In Section 5.4 we consider this question in the $C^k$ regularity spaces (here
$k\in\N$), and show that this regularity is indeed preserved by the flow.
We consider the question of embeddedness in Section 5.7, with the main result
there showing that embedded curves with small length relative to their
chord-arclength ratio will remain embedded.
Since the chord-arclength ratio is scale-invariant but length is not, we
comment that this condition can always be satisfied by rescaling the initial
data.

We summarise this in the following theorem.

\begin{theorem}\label{ckcharch}
Let $k\in\N_0$ be a non-negative integer.
For each $X_0\in \imm^k$ there exists a unique eternal
$H^1(ds)$ curve shortening flow $X:\S^1\times\R\rightarrow\R^2$ in
$C^1(\R; \imm^k)$ such that $X(\cdot,0) = X_0$.

Furthermore, suppose $X_0$ satisfies
\begin{equation}
\inf_{s\in[0,\L_0]} \frac{\C(s)}{\mcs(s)} 
> \frac{L_0^2\sqrt{2+\vn{X_0}_\infty^2}}{4} e^{\frac{L_0^2\sqrt{2+\vn{X_0}_\infty^2}}{4}}
\,.
\label{EQembcondn}
\end{equation}
where $\C$ and $\mcs$ are the chord and arclengths respectively.
Then $X(t)$ is a family of embeddings.
\end{theorem}

\begin{figure}[h]
\includegraphics[width=12cm]{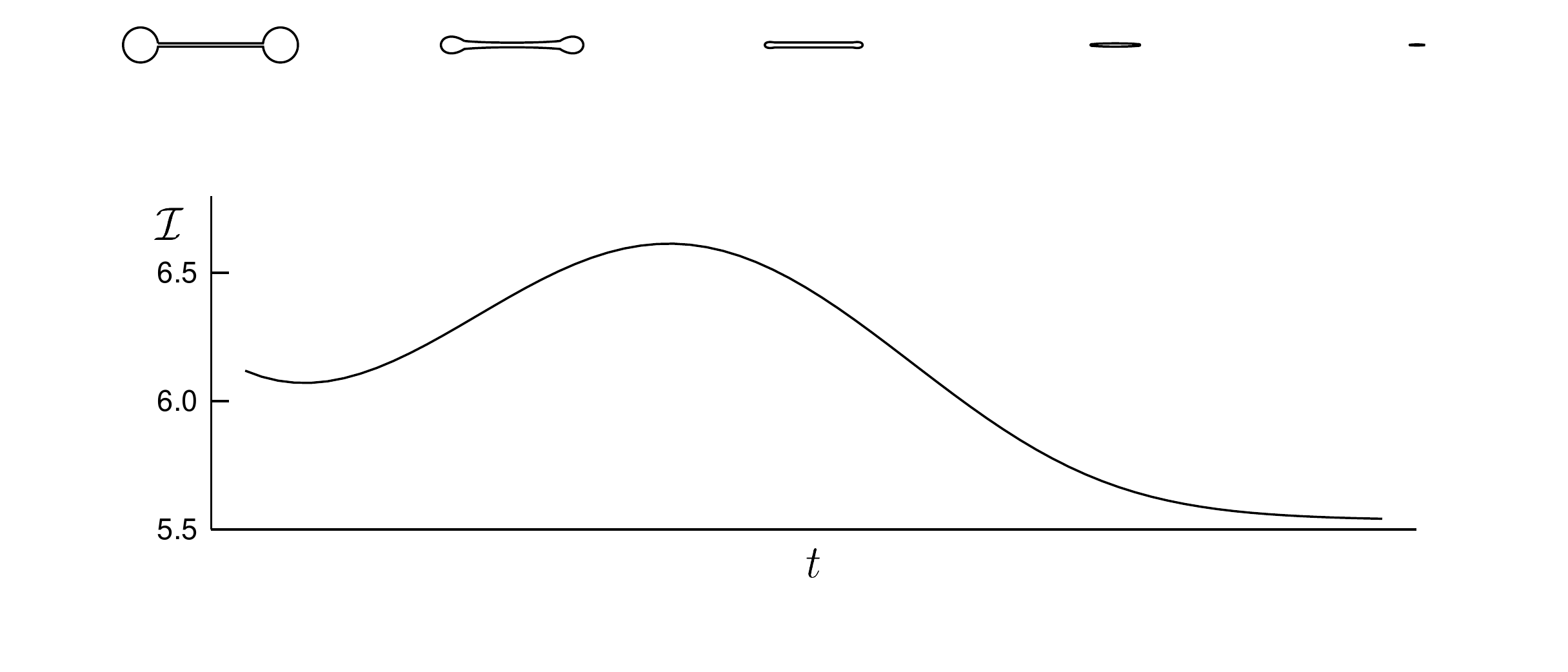}
\caption{Evolution for a barbell inital curve. 70 steps of size 0.1.}
\label{fig:barbell}
\end{figure}

Although the $H^1(ds)$ curve shortening flow disappears (in infinite time), we are interested in
identifying if it asymptotically approaches any particular shape.
In order to do this, we define the \emph{asymptotic profile} of a given
$H^1(ds)$ curve shortening flow $Y:\S\times\R\rightarrow\R^2$ by
\[
 Y(t,u):=e^{t}\left(X(t,u)-X(t,0)\right)\,.
\]
Because of the exponential rescaling bounds for $Y$ and its gradient become more difficult than for $X$. On the other hand, scale invariant estimates for $X$ such as the chord-arc ratio and isoperimetric ratio carry through directly to estimates on $Y$.
Furthermore, for $H^1(ds)$ curve shortening flows on the $C^2$ space the curvature scalar $k$  is well defined, and we can ask meaningfully if curvature remains controlled along $Y$
(on $X$, it will always blow up).

By considering the asymptotic profile we hope to be able to identify limiting profiles $Y_\infty$ for the flow.
For the vanilla flow $X$, the limit is always a constant map.
In stark contrast to this, possible limits for $Y$ are manifold.
We are able to show that the asymptotic profile $Y$ does converge to a unique
limit $Y_\infty$ depending on the initial data $X_0$, but it seems difficult to
classify precisely what these $Y_\infty$ look like.
For the curvature, we show in Section 5.5 that it is uniformly bounded for
$C^2$ initial data, and that the profile limit $Y_\infty$ is immersed with
well-defined curvature (Theorem \ref{TMcurvatureprofile}).
On embeddedness, the same result as for $X$ applies due to scale-invariance.
The isoperimetric deficit $\D_Y$ on $Y$ (in a scale-invariant sense) is studied in Section 5.6.
It isn't true that the deficit is monotone, or improving, but at least we can
show that the eventual deficit of the asymptotic profile limit $Y_\infty$ is
bounded by a constant times the deficit of the initial data $X_0$; this is
sharp.

We summarise these results in the following theorem.

\begin{theorem} \label{besttheorem}
Let $k\in\N_0$ be a non-negative integer.
Set $\SB$ to $H^1(\S^1,\R^2)\setminus\SC$ for $k=0$ and otherwise 
set $\SB$ to $C^k(\S^1,\R^2)\setminus\SC$.
For each $X_0\in\SB$ there exists a non-trivial $Y_\infty\in
H^1(\S^1,\R^2)\setminus\SC$ such that the asymptotic profile $Y(t)\to Y_\infty$ 
in $C^0$ as $t\to\infty$.

Furthermore:
\begin{itemize}
\item $Y_\infty$ is embedded if at any $t\in(0,\infty)$ the condition \eqref{EQembcondn} was satisfied for $X(t)$
\item If $k\ge2$, and $X(0)$ is immersed, then $Y_\infty$ is immersed with bounded curvature
\item There is a constant $c = c(\vn{X(0)}_\infty)$ such that the isoperimetric deficit of $Y_\infty$ satisfies
\[
	\D_{Y_\infty} \le c\D_{X(0)}
	\,.
\]
\end{itemize}
\end{theorem}



\addtocontents{toc}{\protect\setcounter{tocdepth}{0}}

\section*{Acknowledgements}

The first author is grateful to Shinya Okabe and Kazumasa Fujiwara for helpful conversations.

\addtocontents{toc}{\protect\setcounter{tocdepth}{2}}


\section{Metrics on spaces of immersed curves}

Let $C^k(\S^1,\R^2)$ be the usual Banach space of maps with continuous derivatives up to order $k$.
Our convention is that $|\S^1| = |[0,1]| = 1$.
For $ 1\leq k \leq \infty$ we define 
\[
 \imm^k :=\{ \gamma \in C^k(\S^1,\R^2): \enorm{\gamma'(u)}\neq 0 \}\,.
 \]
Note that $\imm^k$ is an open subset of $C^k(\S^1,\R^2)$.

The tangent space $T_{\gamma}\imm^k\cong C^{k}(\S^1,\R^2)$
consists of vector fields along $\gamma$. We define the following Riemannian
metrics on $\imm^1$ for $v,w\in T_\gamma \imm^1$:
\begin{align}\nonumber
\ip{v,w}_{L^2} &:= \int_0^1 \ip{v,w} du\\ \nonumber
\ip{v,w}_{H^1} &:= \ip{v,w}_{L^2}+\ip{v',w'}_{L^2}\\ \nonumber
\ip{v,w}_{L^2(ds)} &:= \int_0^{\L(\gamma)}\ip{v,w} \, ds \\
\ip{v,w}_{H^1(ds)}&:=\ip{v,w}_{L^2(ds)}+\ip{v_s,w_s}_{L^2(ds)} \label{h1dsprod}
\end{align}
The length function \eqref{eq:length} is of course well-defined on the larger Sobolev space $H^1(\S^1,\R^2)$, as are the $L^2, H^1$ and $L^2(ds)$ products above. However, the $H^1(ds)$ product is not well-defined because of the arc-length derivatives, even if one restricts to curves which are almost everywhere immersed.

We remark that the $L^2(ds)$ metric on $\imm^1$ is an example of a \emph{weak} Riemannian metric - a purely infinite dimensional phenomenon where the topology induced by the Riemannian metric is weaker than the manifold topology. In fact, even for a strong Riemannian metric, geodesic and metric completeness are not equivalent (as guaranteed by the Hopf-Rinow theorem in finite dimensions) and it is not always the case that points can be joined by minimising geodesics (for an overview of these and related facts see \cite{Bruveris:2016aa}).

\subsection{Vanishing Riemannian distance in $(\IM,L^2(ds))$}
\label{SNawesome}
Consider curves $\gamma_1, \gamma_2 \in \IM$ and a smooth path $\alpha:[0,1]\rightarrow \IM$ with $\alpha(0) = \gamma_1$ and $\alpha(1) = \gamma_2$.
The $L^2(ds)$-length of this smooth path is well-defined and given by
\begin{equation}
\label{EQdefnl2ds}
	\SL_{L^2(ds)}(\alpha)
	:= \int_0^1 \norm{\alpha'(t)}_{L^2(ds)} \, dt
	= \int_0^1 \left(\int_0^{\L(\alpha(t))} \enorm{ \alpha_t(t,s)}^2 \,ds\right )^{\tfrac{1}{2}}\, dt 
\,.
\end{equation}
As usual, one defines a distance function associated with the Riemannian metric by
\[ \dst(\gamma_1,\gamma_2):=\inf\{ \SL_{L^2(ds)}(\alpha): \alpha \text{ piecewise smooth path from } \gamma_1 \text{ to } \gamma_2 \}
\]
Since the $L^2(ds)$ metric is invariant under the action of $\text{Diff}( \S^1)$ it induces a Riemannian metric on
the quotient space $\SQ$ (except at the singularities) as follows. Let $\pi:\IM\to \SQ$ be the projection. Given $v,w\in T_{[\gamma]}\SQ$ choose any $V,W\in T_\gamma\IM$ such that $\pi(\gamma)=[\gamma],$  $T_\gamma \pi (V)=v, T_\gamma\pi(W)=w$. Then the quotient metric is given by
\[ \ip{v,w}_{[\gamma]}:=\ip{V^{\perp},W^\perp}_{\gamma,L^2(ds)} \]
where $V^\perp$ and $W^\perp$ are projections onto the subspace of $T_\gamma\IM$ consisting of vectors which are tangent to the orbits. This is just the space of vector fields along $\gamma$ in the direction of the normal $N$ to $\gamma$, and so $V^\perp=\ip{V,N}N$. The length of a path $\pi(\alpha)$ in $\SQ$ according to the quotient metric is then
\[\SL^\SQ_{L^2(ds)}(\pi(\alpha))=\int_0^1\left (\int_0^\SL \enorm{\alpha^\perp_t}^2\,  ds\right)^\frac{1}{2}\, dt\ \]
and the distance is
\begin{align*}
	\dst^\SQ([\gamma_1],[\gamma_2])&= \inf \{ \L_{L^2(ds)}^\SQ(\pi(\alpha)): \pi(\alpha) \text{ piecewise smooth path from } [\gamma_1] \text{ to } [\gamma_2] \}
\end{align*}
This is the distance function that Michor and Mumford have shown to be identically zero (Theorem \ref{TMMM}). 
They also point out (cf.  \cite[Section 2.5]{MR2201275}) that for any smooth path $\alpha$ between curves $\gamma_1$, $\gamma_2$, there exists a
smooth $t$-dependent family of reparametrisations $\phi:[0,1]\to
\text{Diff}(\S^1)$ such that the reparametrised path\footnote{Note that as paths in the full space $\IM$, $\tilde \alpha$ is different to $\alpha$, but they project to the same path in $\SQ$.} $\tilde \alpha(t,u)
:=\alpha(t,\phi(t,u))$ has path derivative $\tilde \alpha_t(t)$ which is \emph{normal} to $\tilde \alpha(t)$ . Thus an equivalent definition is
\[\dst^\SQ([\gamma_1],[\gamma_2])=\inf \{\SL_{L^2(ds)}(\alpha)\,:\,\alpha \text{ p.w. smooth with } \alpha(0)\in[\gamma_1], \alpha(1)\in [\gamma_1]\} \]

\begin{theorem}[Michor-Mumford \cite{MR2201275}]\label{TMMM}
For any $\varepsilon>0$ and $[\gamma_1], [\gamma_2]$ in the same path component of $\SQ$ there is a path $\alpha:[0,1]\to \imm^\infty $ satisfying $\alpha(0)\in[\gamma_1], \alpha(1)\in [\gamma_2]$ and having length $\L_{L^2(ds)}(\alpha)<\varepsilon $.
\end{theorem}
Since it is quite a surprising result and an elegant construction, we include a description of the proof. 
The idea is to show that we may deform any path $\alpha$ in $\imm^\infty$ to a new path $\alpha_n$ that remains smooth but has
small normal projection, and whose endpoint changes only by reparametrisation.

So, let us consider a smooth path $\alpha:[0,1]\rightarrow\imm^\infty$ such that $\alpha(0) = \gamma_1$ and $\alpha(1) = \gamma_2$.
We choose evenly-spaced points $\theta_0,\ldots,\theta_n$ in $\S^1$ and 
move $\gamma_1(\theta_i)$, via $\alpha(2t)$, to their eventual destination
$\gamma_2(\theta_i)$ twice as fast.  The in-between points $\psi_i =
(\theta_{i-1} + \theta_i)/2$ should remain stationary while this occurs.  Once
half of the time has passed, and all points $\theta_i$ are at their
destination, the points $\gamma_1(\psi_i)$ may begin to move via $\alpha$.  They
should also move twice as fast as before.
A graphical representation of this is given in Figure 3.
\input{mmfigure2}

The resultant path $\alpha_n$ has small normal projection (depending on $n$) but
also longer length (again depending on $n$).
The key estimate in \cite[Section 3.10]{MR2201275} shows that the length of the
path $\alpha_n$ increases proportional to $n$ and the normal projection decreases
proportional to $\frac1n$.
Since the normal projection is squared, this means that the length of $\alpha_n$ is
proportional to the length of $\alpha$ times $\frac{1}{n}$.

In other words, for any $\varepsilon>0$ we can connect $\gamma_1$ to a \emph{reparametrisation} of
$\gamma_2$ by a path with $L^2(ds)$-length less than $\varepsilon$, and so the distance
between them is zero.
\footnote{We remark that this phenomenon of triviality of the metric topology
induced by the Riemannian $L^2(ds)$ metric on the quotient space is also
established in higher dimensions, see \cite{MM05}.}

Note that the corresponding result does not immediately follow for the full space $\IM$ because in the full space the tangential component of the path derivative is also measured. Indeed, we could apply the Michor-Mumford construction to obtain a path $\alpha_n$ whose derivative has small normal component, and then introduce a time-dependent reparametrisation to set the tangential component to zero, but  of course the reparametrisation changes the endpoint to a reparametrisation of the original endpoint. However, as we show in the following theorem, it is still possible to get the desired result by diverting through a sufficiently small curve.

\begin{theorem}\label{TMawesome}
The $L^2(ds)$-distance between any two curves in the same path component of $\IM$ vanishes.
\end{theorem}
\begin{proof}
Let $\gamma_0,\gamma_1 \in C^\infty(\S^1,\R^2)$ be smooth immersions in the same path component, and let $x$ be another curve in the same component with $\norm{x'}_{L^\infty}<\left(\frac{\varepsilon}{4}\right)^{2/3}$ (for example, $x$ could be a sufficiently small scalar multiple of $\gamma_0$). By Theorem \ref{TMMM} there exists a path $\alpha_0$ from $\gamma_0$ to $y$, where $y$ is a reparametrization of $x$, with $\SL_{L^2(ds)}(\alpha_0)<\frac{\varepsilon}{4}$. Now let $\theta\in\text{Diff}( \S^1)$  such that $y(u)=x(\theta(u))$ and define a path $\alpha_1$ from $x$ to $y$ by
\begin{align*}
	c(t,u)&:=(1-t)u+t\theta(u)\\
\alpha_1(t,u)&:=x(c(t,u))
\end{align*}
Then
\begin{align*}
	\partial_t\alpha_1 &=x'(c)c_t=x'(c)(\theta(u)-u)\\
\partial_u\alpha_1 &=x'(c)c_u
\end{align*}
and by \eqref{EQdefnl2ds} the $L^2(ds)$-length of $\alpha_1$ is
\begin{align*}
	\SL_{L^2(ds)}(\alpha_1)=\int_0^1\left (\int_{ \S^1}\enorm{x'(c)}^3|\theta(u)-u|^2 \, dc\right)^{1/2}\, dt
\leq \norm{x'}_{L^\infty}^{3/2}\leq \frac{\varepsilon}{4}
\end{align*}
Now we concatenate $\alpha_0$ with $\alpha_1(-t)$ to form a path $p$ from $\gamma_0$ to $x$ with $\SL_{L^2(ds)}(p)<\frac{\varepsilon}{2}$. By the same method we construct a path $q$ from $\gamma_1$ to $x$ with $\SL_{L^2(ds)}(q)<\frac{\varepsilon}{2}$ and then the concatenation of $p$ with $q(-t)$ is a path from $\gamma_0$ to $\gamma_1$ with arbitrarily small $L^2(ds)$-length. We assumed $\gamma_0,\gamma_1$ were smooth, but since the smooth curves are dense in $C^1$ and $\norm{\alpha_t}_{L^2(ds)}\leq c\norm{\alpha_t}_{C^1}\norm{\alpha_u}_{L^2} $ for any path $\alpha$, we can also join any pair of curves in $\imm^1$ by a path with arbitrarily small $L^2(ds)$-length.	
\end{proof}

\begin{remark}
As mentioned in the introduction, an alternative proof of a more general result is outlined in \cite{Bauer:2012aa}. The proof relies on another theorem of Michor and Mumford \cite{MM05} (extended in eg. \cite{BHP20}) showing that the right invariant $L^2$ metric on diffeomorphism groups gives vanishing distance.
\end{remark}

\section{Symmetries of metrics and gradient flows} \label{symmetries}

The standard curve shortening flow \eqref{eq:csf} enjoys several important symmetries:
\begin{itemize}

\item \emph{Isometry of the plane:} if $A:\R^2\to\R^2$ is an isometry and $X:\S^1\times[0,T)\to \R^2$ is a solution to curve shortening flow then $A\circ X$ is also a solution.

\item \emph{Reparametrisation:} if $\phi\in \text{Diff}(\S^1)$ and $X(u,t)$ is a solution to curve shortening flow then $X(\phi(u),t)$ is also a solution.

\item \emph{Scaling spacetime:} if $X(u,t)$ is a solution to curve shortening flow then so is $\lambda X(u,t/\lambda^2), 
$ with $\lambda>0$.
\end{itemize}

It is interesting to note that these symmetries can be observed directly from symmetries of the length functional $\L$ and the $H^{0}(ds)$ Riemmannian inner product without actually calculating the gradient. 

\begin{lemma}\label{LMsym}
Suppose there is a free group action of $G$ on $(M, g)$ which is an isometry of the Riemannian metric $g$ and which leaves $E:M\to \R$ invariant. Then the gradient flow of $E$ with respect to $g$ is invariant under the action. 
\end{lemma}

\begin{proof}
Since $E(x)=E(\lambda x)$ for all $\lambda \in G$ we have
\[ dE_x=dE_{\lambda x}d\lambda_x \]
then equating
\begin{align*}
dE_xV &=\langle \grad E_x,V\rangle_x=\langle d\lambda_x\grad E_x,d\lambda_xV\rangle_{\lambda x}\\
dE_{\lambda x}d\lambda_x V &=\langle \grad E_{\lambda x},d\lambda_x V\rangle_{\lambda x}
\end{align*}
shows that $d\lambda_x\grad E_x=\grad E_{\lambda x}$ ($d\lambda_x$ has full rank because the action is free). Therefore if $X$ is a solution to $X_t=-\grad E_{X}$ then 
\[ (\lambda X)_t=-d\lambda_X \grad E_X=-\grad E_{\lambda X} \]
so $\lambda X$ is also a solution.
\end{proof}

To demonstrate we observe the following symmetries of the $H^1(ds)$ gradient flow.

\emph{Isometry}. 
An isometry $A:\R^2\to \R^2$ induces $A:\imm^1(\S^1,\R^2)\to \imm^1(\S^1,\R^2)$ by $A\gamma=A\circ \gamma$. Since an isometry is length preserving
\begin{align}\nonumber
\L(A\gamma)&=\L(\gamma)\\ 
\implies d\L_{A\gamma}dA_\gamma &=d\L_\gamma \label{lengthequivalence}
\end{align}
and similarly for the arc-length functions $s_{A\gamma}=s_\gamma \implies ds_{A\gamma}=ds_\gamma$. Hence, in the $H^1(ds)$ metric:
\begin{align*}
	\ip{dA_\gamma(\xi),dA_\gamma(\eta)}_{A\gamma}&= \int \langle dA_\gamma(\xi),dA_\gamma(\eta)\rangle ds_{A\gamma}+\int \langle \tfrac{d}{ds}dA_\gamma(\xi), \tfrac{d}{ds} dA_\gamma(\eta)\rangle ds_{A\gamma}\\
&= \ip{\xi,\eta}_\gamma
\end{align*}
i.e. the induced map $A:\imm^1(\S^1,\R^2)\to \imm^1(\S^1,\R^2)$ is an $H^1(ds)$ isometry. Now by Lemma \ref{LMsym} if $X$ is a solution of the $H^1(ds)$ gradient flow of $\L$ then so is $AX$.

\emph{Reparametrisation.}
Given $\phi \in \diff(\S^1)$ we have $\L(\gamma)=\L(\gamma\circ \phi)$ and the map $\Phi(\gamma)=\gamma\circ \phi$ is linear on $\imm^k(\S^1,\R^2)$ so $d\L_\gamma=d\L_{\Phi\gamma}\Phi$. Assuming for simplicity that $\phi'>0$ we have 
\begin{align*}
\ip{\Phi \xi,\Phi \eta}_{\Phi\gamma} &=
\int_{\S^1}\ip{\xi( \phi(u)),\eta(\phi(u))}|\gamma'(\phi)\phi'(u)|du\\
& \qquad + \int_{\S^1}\ip{\tfrac{1}{|y'(\phi)\phi'(u)|}\tfrac{d}{du}\xi( \phi(u)),\tfrac{1}{|y'(\phi)\phi'(u)|}\tfrac{d}{du}\eta(\phi(u))}|\gamma'(\phi)\phi'(u)|du\\	
&=\int_{\S^1}\ip{\xi( \phi),\eta(\phi)}|\gamma'(\phi)|d\phi\\
& \qquad + \int_{\S^1}\ip{\tfrac{1}{|y'(\phi)|}\tfrac{d}{d\phi}\xi( \phi),\tfrac{1}{|y'(\phi)|}\tfrac{d}{d\phi}\eta(\phi)}|\gamma'(\phi)|d\phi\\	
&=\ip{\xi,\eta}_\gamma
\end{align*}
so $\Phi$ is also an $H^1(ds)$ isometry and again by Lemma \ref{LMsym} the gradient flow is invariant under reparametrisation.

\emph{Scaling space-time.} 
For a dilation of $\R^2$ by $\lambda>0$ the length function scales $\L(\lambda x)=\lambda \L(x)$. To get a space-time scaling symmetry we need the metric to also be homogeneous. However the $H^1(ds)$ metric is not homogeneous:
\begin{align*}
 \ip{\lambda \xi,\lambda \eta}_{\lambda \gamma}&=
\int \ip{\lambda \xi,\lambda \eta}\lambda ds_\gamma+\int \ip{\frac{1}{\lambda}\frac{d}{ds_\gamma}\lambda \xi,\frac{1}{\lambda}\frac{d}{ds_\gamma}\lambda \eta}\lambda ds_\gamma \\
&= \lambda^3\int\ip{\xi,\eta}ds+\lambda\int\ip{\xi_s,\eta_s}ds
\end{align*}
If we want space-time scaling we need to use a different metric. For example the metric
\begin{align}
\ip{\xi,\eta}_{H^1(d\bar s)} :=\ip{h,k}_{L^2(d\bar s)}+\L^2\ip{\xi_s,\eta_s}_{L^2(d\bar s)} 	\label{eq:SYM1}
\end{align}
which is used in \cite{Sundaramoorthi2007} satisfies $\ip{\lambda \xi,\lambda \eta}_{\lambda\gamma}=\lambda^2\ip{\xi,\eta}_\gamma$, and then 
\[d\L_{\lambda\gamma} \xi=\ip{\grad\L_{\lambda \gamma}, \xi}_{\lambda\gamma}= \ip{\grad\L_{\lambda \gamma},\xi}_{\gamma}\]
\[d\L_{\gamma} \xi=\ip{\grad\L_{\gamma},\xi}_{\gamma} \]
Since $\L(\lambda \gamma)=\lambda \L(\gamma)$ we have 
\[ d\L_{\lambda \gamma}\lambda v= \frac{\partial}{\partial t}\L(\lambda \gamma+t\lambda v)=\lambda d\L_\gamma v \]
 i.e. $d\L_{\lambda \gamma} = d\L_\gamma $ and so from above $\grad\L_{\lambda \gamma}=\grad\L_{\gamma}$. Now if $X(u,t)$ is a solution to $X_t=-\grad \L_X$ then defining $\tilde X(u,t):=\lambda X(u,t/\lambda)$ we have
\[ \tilde X_t(u,t)=X_t(u,t/\lambda)=-\grad\L_{X(u,t/\lambda)}=-\grad\L_{\tilde X} \]



\section{The gradient flow for length with respect to the $H^1(ds)$ Riemannian  metric}

In this section our focus is on the $H^1(ds)$-gradient flow for length on a variety of spaces.

\subsection{Derivation, stationary solutions and circles}

The $H^1(ds)$ gradient of length is defined by
\begin{align*}
d\L_\gamma V=\ip{\grad\L_\gamma,V}_{H^1(ds)}
&=\int \ip{ \grad\L_\gamma,V}ds+\int \ip{ (\grad\L_\gamma)_s,V_s}ds\\
&=\int\ip{\grad\L_\gamma -(\grad\L_\gamma)_{ss},V}ds
\end{align*}
Comparing with \eqref{eq:dlength} the gradient of length with respect to the $H^1(ds)$ metric must satisfy weakly
\begin{equation}\label{eq:gradODE}
	(\grad\L_\gamma)_{ss}-\grad\L_\gamma=\frac{dT}{ds}
\,.
\end{equation} 
We solve this ODE in arc-length parametrisation using the Green's function method. Considering 
\begin{equation}\label{eq:fundamental}
G_{ss}(s,\tilde s)-G(s,\tilde s)=\delta(s-\tilde s)
\end{equation}
(again weakly) with $C^1$-periodic boundary conditions and the required discontinuity we find the Green's function
\begin{equation}\label{eq:greens1}
 G(s,\tilde s)=\frac{\cosh\left (|s-\tilde s|-\tfrac{\L}{2}\right)}{2\sinh(-\tfrac{\L}{2})}
\,.
 \end{equation}
(cf. \cite{Sundaramoorthi2007} eqn. (12) for the metric \eqref{eq:SYM1} above.) 
Then the solution to \eqref{eq:gradODE} is 
\begin{align}\label{eq:h1grad}
\grad\L_\gamma (s)=\int_0^\L \frac{dT}{d\tilde s}G(s,\tilde s)d\tilde s 
\,.
\end{align}
 We can integrate by parts twice in \eqref{eq:h1grad} to obtain
\[ \grad\L_\gamma (s)=\gamma(s)+\int_0^\L\gamma(\tilde s)G(s,\tilde s) d\tilde s \]
and we observe that it is not neccesary for $\gamma$ to have a second derivative. Indeed, using integration by parts and \eqref{eq:fundamental} we find
\begin{align*}
\ip{\grad\L_\gamma,V}_{H^1(ds)}&=\int \ip{\gamma+\int \gamma G\,d\tilde s,V} + \ip{\gamma_s+\int \gamma G_s \, d\tilde s ,V_s} \, ds \\
&= \int \ip{\gamma+\int \gamma (G-  G_{ss})\, d\tilde s,V}+\int \ip{\gamma_s,V_s}\, ds\\
&=\int \ip{\gamma_s,V_s}\, ds\\
&= d\L_\gamma V
\end{align*}

\begin{definition}

Consider a family of curves $X:\S^1\times(a,b)\rightarrow\R^2$ where for each $t\in(a,b)\subset \R$, $X(\cdot,t)\in \imm^1$. 
We term $X$ an \emph{$H^1(ds)$ curve shortening flow} if 
\begin{equation}\label{eq:h1csf}
 \partial_t X(s,t)=-X(s,t)-\int_0^\L X(\tilde s,t)G(X;s,\tilde s)d\tilde s
\end{equation}
where $G$ is given by
\begin{equation}
\label{EQdefG}
G(X;s,\tilde s)=\frac{\cosh\left (|s-\tilde s|-\tfrac{\L}{2}\right)}{2\sinh(-\tfrac{\L}{2})} \quad \text{ for } 0\leq s,\tilde s \leq \L
\,.
\end{equation}
Here we write $G(X;s,\tilde{s})$ to emphasize the dependence on the curve $X(.,t)$, but henceforth we will omit the first argument unless it is needed to avoid ambiguity.
\end{definition}

\begin{remark}\label{flowspace}
	Note that \eqref{eq:h1csf} makes sense on the larger space $H^1(\S^1,\R^2)\setminus\SC$ where
\[
\SC := \{X\in H^1(\S^1,\R^2)\,:\,\vn{X'(u)}_{L^2} = 0\}
\]
is the space of constant maps, provided we do not use the arc length parametrisation. That is, we consider
\begin{equation}\label{eq:ode}
\partial_t X(u,t)=F(X(u,t))\end{equation}
where $F$ is defined by
\begin{equation}
\label{EQgenldefs}
\begin{split}
	F(x;u)&:= -x(u)-\int_0^1 x(\tilde u)G(x;u,\tilde u)|x'(\tilde u)|\,d\tilde u \\
	G(x;u,\tilde u)&:= \frac{\cosh\left( |s_x(u)-s_x(\tilde u)|-\tfrac{\L(x)}{2} \right)}{2 \sinh\Big(-\tfrac{\L(x)}{2}\Big)}
\end{split}
\end{equation}
The constant maps are problematic: viewing $G$ as a map from $H^1(\S^1,\R^2)\times\S^1\times\S^1 \rightarrow\R$ we see that taking a sequence in the first variable toward the space of constant maps results in $-\infty$.
Then in the evolution equation \eqref{eq:h1csf}, the integral involving $G$ along such a sequence is not well-defined.
Most of the results that follow will be proved for this larger space $H^1(\S^1,\R^2)\setminus\SC$.  However, the interpretation of this flow as the $H^1(ds)$ gradient flow of length requires that we use the space $\imm^1$. This is so that $H^1(ds)$ is a Riemannian metric: the product \eqref{h1dsprod} is not positive definite at curves which are not immersed, and in fact is not necessarily well-defined because of the arc length derivatives. Moreover,  $\L$ is not differentiable outside of $\imm^1$. Nevertheless we proceed to study the flow mostly in the space $H^1(\S^1,\R^2)\setminus\SC$, but bear in mind that on this space it is some kind of pseudo-gradient. 
\end{remark}

We begin our study of the flow by considering stationary solutions and observing the evolution of circles.

\begin{lemma}
\label{LMstat}
There are no stationary solutions to the $H^1(ds)$ curve shortening flow.
\end{lemma}
\begin{proof}
From \eqref{eq:h1csf}, a map $X\in H^1(\S^1,\R^2)$ is stationary if 
\[
	X(s) = -\int_0^\L X(\tilde s)G(s,\tilde s)d\tilde s
\,.
\]
The arc-length function is in $H^1(\S^1,\R)$ and so $G$ (see \eqref{EQdefG}) is in turn in $H^1(\S^1,\R)$.

Differentiating, we find
\[
	X_s(s) = -\int_0^\L X(\tilde s)G_s(s,\tilde s)d\tilde s
\,,
\]
and so the first derivative of $X$ exists classically. Iterating this with integration by parts shows that in fact all derivatives of $X$ exist and it is a smooth map.

Furthermore, examining the case of the second derivative in detail, we find (applying \eqref{eq:fundamental})
\begin{equation}
\label{EQpoint}
X_{ss}=-\int_0^\L X(\tilde s)G_{ss}(s, \tilde s)d\tilde s=-X(s)-\int_0^\L X(\tilde s)G(s,\tilde s)d\tilde s=0
\,.
\end{equation}
Since $X$ is periodic, this implies that $X$ must be the constant map. As explained in remark \ref{flowspace}, $G$ is singular at constant maps. 
\end{proof}


Let us now consider the case of a circle. Here, we see a stark difference to the case of the classical $L^2(ds)$ curve shortening flow.
\begin{lemma}
Under the $H^1(ds)$ curve shortening flow, an initial circle in $H^1(\S^1,\R^2)$ with any radius and any centre will
\begin{enumerate}
\item[(i)] Exist for all time; and
\item[(ii)] Shrink homothetically to a point as $t\rightarrow\infty$.
\end{enumerate}
\end{lemma}
\begin{proof}
We can immediately conclude from the symmetry of the length functional and the symmetry of the circle that the flow must evolve homothetically (see Lemma \ref{LMsym}).
We must calculate the evolution of the radius of the circle.

So, suppose that $X$ is an $H^1(ds)$ curve shortening flow of the form 
\begin{equation}\label{eq:shrinker}
	X(u,t)=r(t)(\cos u ,\sin u)
\end{equation}
with $r(0)>0$. Here $u$ is the arbitrary parameter and not the arclength variable.
Then $X(s,t)=r(t)(\cos(\frac{s}{r}),\sin(\frac{s}{r}))$ and $X_{ss}=-\frac{1}{r^2}X$.
Therefore
\begin{align*}
X_t(s,t)&=-\int X_{\tilde s\tilde s}(\tilde s,t)G(s, \tilde s)d\tilde s 
=\frac{1}{r^2}\int X(\tilde s,t)G(s, \tilde s)d\tilde s
\,.
\end{align*} 
Applying \eqref{eq:fundamental} and integrating by parts gives
\begin{align*}
X_t&=-\frac{1}{r^2}X_t-\frac{1}{r^2}	X
\,.
\end{align*} 
Differentiating \eqref{eq:shrinker} gives $X_t=\frac{\dot r}{r}X$ 
and then substituting into the above leads to the ODE for $r(t)$: 
\begin{align*}
\dot r=-\frac{r}{r^2+1} 
\,.
\end{align*}
Using separation of variables yields
$ r^2e^{r^2}=e^{-2t+c}$ (here $c = \log(r^2(0)e^{r^2(0)})$)
which has solutions
\[ r(t)=\pm \sqrt{W(e^{c-2t})} \]
where $W$ is the Lambert $W$ function (the inverse(s) of $xe^x$). 
Since $t\mapsto\sqrt{W(e^{c-2t})}$ is a monotonically decreasing function converging to zero as $t\rightarrow\infty$, this finishes the proof.
\end{proof}


\subsection{Existence and uniqueness}\label{ssexistence}

Now we turn to establishing existence and uniqueness for the $H^1(ds)$ curve
shortening flow.

For the initial data, we take it to be on the largest possible space for which \eqref{eq:h1csf} makes sense . As explained in Remark \ref{flowspace}, 
this is the space of maps $H^1(\S^1,\R^2)\setminus\SC$ where
\[
\SC := \{X\in H^1(\S^1,\R^2)\,:\,\vn{X'(u)}_{L^2} = 0\}
\]
is the space of constant maps.
We note that $\SC$ is generated by the action of translations in $\R^2$ applied
to the orbit of the diffeomorphism group at any particular constant map in
$H^1(\S^1,\R^2)$.
Since the orbit of the diffeomorphism group applied to a constant map is
trivial, the space $\SC$ turns out to be two-dimensional only.

The main result of this section is the following.

\begin{reptheorem}{TMglob}
For each $X_0\in H^1(\S^1,\R^2)\setminus\SC$ there exists a unique eternal
$H^1(ds)$ curve shortening flow $X:\S^1\times\R\rightarrow\R^2$ in
$C^1(\R; H^1(\S^1,\R^2)\setminus\SC)$ such that $X(\cdot,0) = X_0$.
\end{reptheorem}

This is proven in two parts.

\subsubsection{Local existence}

We begin with a local existence theorem.

\begin{theorem}\label{shortexistence}
For each $X_0\in H^1(\S^1,\R^2)\setminus\SC$ there exists a $T_0>0$ and unique $H^1(ds)$
curve shortening flow $X:\S^1\times[-T_0,T_0]\rightarrow\R^2$ in
$C^1([-T_0,T_0]; H^1(\S^1,\R^2)\setminus\SC)$ such that $X(\cdot,0) = X_0$.
\end{theorem}

The flow \eqref{eq:h1csf} is essentially a first-order ODE and so we will be able to establish this result by applying the Picard-Lindel\"of theorem in $H^1(\S^1,\R^2)\setminus\SC$.
(see \cite[Theorem 3.A]{Zeidler:1986aa}).
Note that this means we should not expect any kind of smoothing effect or other phenomena associated with diffusion-type equations such as the $L^2(ds)$ curve shortening flow.
Of course, we will need to show that the flow a-priori remains away from the problematic set $\SC$.

Recalling \eqref{EQgenldefs}, we observe the following regularity for $F$ in our setting.

\begin{lemma}
\label{LMh1f}
For any $x\in H^1(\S^1,\R^2)\setminus\SC$ consider $F$ and $G$ as defined in \eqref{EQgenldefs}.
Then $F(x) \in H^1(\S^1,\R^2)$.
\end{lemma}
\begin{proof}
The weak form of equation \eqref{eq:fundamental} implies continuity and symmetry of the Greens function $G$, as well as
\[
\int_0^\L G(s,\tilde s)\, ds =-1
\,.
\]
Note that there is a discontinuity in the first derivative of $G$ with respect to either variable.
Since $G$ is strictly negative we have $\int_0^\L |G(s,\tilde s)| d\tilde s=1$. 
Now
\begin{align*}
	F(x)&= -x-\int_0^\L x(\tilde s)G(s,\tilde s)\, d\tilde s\\
& =\int_0^\L (x(s)-x(\tilde s))G(s,\tilde s) \, d\tilde s \\
&=\int_0^\L \int_{\tilde s}^s x_s \, ds \, G(s,\tilde s) \, d\tilde s
\end{align*}
therefore
\begin{align}\label{eq:gradbound0}
|F(x)|\leq \int_0^\L |s-\tilde s||G(s,\tilde s)| \, d\tilde s \leq \L(x) 
\end{align}
and we have 
\begin{equation}\label{eq:Fbound0}
\norm{F(x)}_{L^2(du)}\leq \L(x)	
\,.
\end{equation}
For the derivative with respect to $u$, using $G_s=-G_{\tilde s}$ and integration by parts we find
\begin{align}\nonumber
F(x;u)_u&=-x_u-\enorm{x_u}\int_0^\L x(\tilde s)G_s(s,\tilde s) \, d\tilde s \\
&= -x_u-\enorm{x_u}\int_0^\L x_{\tilde s}(\tilde s)G(s,\tilde s) \, d\tilde s\,.   \label{eq:fu}
\end{align}
This implies 
\begin{equation}\label{eq:Fbound1}
	\norm{F(x)_u}_{L^2} \leq 2\norm{x_u}_{L^2}
\end{equation}
Since our convention is that $|\S^1| = 1$, we have $\L(x)\leq \norm{x}_{L^2}$, and so the inequalities \eqref{eq:Fbound0} and \eqref{eq:Fbound1} together show that $F(x)\in H^1(\S^1,\R^2)$.
\end{proof}

The $H^1$ regularity of $F$ from Lemma \ref{LMh1f} is locally uniform (for given initial data), with a Lipschitz estimate in $H^1$, as the following lemma shows.

\begin{lemma} \label{blip}
Given 
$x_0\in H^1(\S^1,\R^2)\setminus\SC$ let
\[
Q_b=\{x\in H^1(\S^1,\R^2)\setminus\SC\ :\  {\Vert x-x_0\Vert}_\h1 \leq b<\L(x_0) \}
\]
where $b>0$ is fixed. 
Then there exist constants $L\geq 0$ and $K>0$ depending on $\vn{x_0}_{H^1}$ such that 
\begin{align*}
	\norm{F(x)}_\h1  &< K, && \text{for all} \, x\in Q_b \\
\norm{F(x)-F(y)}_\h1  &\leq L \norm{x-y}_\h1, && \text{for all} \, x,y\in Q_b
\end{align*}
\end{lemma}
\begin{proof}
To obtain the estimates and remain away from the problematic set $\SC$ it is
necessary that the length of each $x\in Q_b$ is bounded away from zero. This is
the reason for the upper bound on $b$. Indeed if $x\in Q_b$ then (note that
$|\S^1| = 1$ in our convention)
\begin{align*}
\big\vert \norm{x'}_{L^1}-\norm{x'_0}_{L^1} \big\vert \leq 
\norm{x'-x_0'}_{L^1}&\leq \norm{x'-x_0'}_{L^2}\leq \norm{x-x_0}_{H^1} \leq b
\end{align*}
hence
\begin{align}
\L(x_0)-b \leq \L(x)\leq \L(x_0)+b \label{eq:lengthbound}
\end{align}
and $\L(x_0)-b>0$ by assumption. It follows that $G(x;u,\tilde u)$ exists on $Q_b$ and since $s_x(u)\leq \L(x)$ we deduce
\begin{equation}\label{eq:Gest1}
|G(x;u,\tilde u)|\leq \frac{\cosh(\tfrac{\L(x)}{2})}{2\sinh(\tfrac{\L(x)}{2})} = \frac{1}{2} \coth\left(\tfrac{\L(x)}{2}\right)
\,.
\end{equation}
We will also need the derivative
\begin{equation}\label{eq:dG}
\partial_u G(x;u,\tilde u)=  \frac{\sinh\left ( |s_x(u)-s_x(\tilde u)|-\tfrac{\L(x)}{2} \right )}{2 \sinh(-\tfrac{\L(x)}{2})}\sign(u-\tilde u)|x'(u)|
\end{equation}
which obeys the estimate
\begin{equation}\label{eq:Gest2}
 \enorm{\partial_u G(x;u,\tilde u)}\leq \frac{1}{2}|x'(u)| 
\,.
\end{equation}
Since $G$ is well-defined on $Q_b$, we may use \eqref{eq:Fbound0} and \eqref{eq:Fbound1} from the proof of Lemma \ref{LMh1f} to obtain
\begin{equation} \label{fboundh1}
	 \norm{F(x)}_{H^1}\leq c\norm{x}_{H^1} \leq cb+c\norm{x_0}_{H^1}=:K
\end{equation}
for a constant $c>0$ and all $x\in Q_b$.  
  
As for the Lipschitz estimate, we will begin by studying the Lipschitz property for $G$.
First note that the arc length function is Lipschitz as a function on $H^1(\S^1,\R^2)$:
\[
\enorm{s_x(u)-s_y(u)}\leq\int_0^u \enorm{\, |x'(a)|- |y'(a)|\,}\,da
\leq \int_0^u\enorm{x'(a)-y'(a)}\,da
\leq \norm{x-y}_\h1
\]
and setting $u=1$ we also have
\[
 \enorm{\L(x)-\L(y)}\leq \norm{x-y}_\h1 \,.
\] 
For the numerator of $G$, note that $\cosh$ is smooth and its domain here is bounded via \eqref{eq:lengthbound}, and so there is a $c_1>0$ such that
\begin{align*}
	&\hskip-2cm\enorm{\cosh \left (\enorm{s_x(u)-s_x(\tilde u)}-\frac{\L(x)}{2}\right)-\cosh \left (\enorm{s_y(u)-s_y(\tilde u)}-\frac{\L(y)}{2}\right) }\\
&\leq c_1\enorm{\enorm{s_x(u)-s_x(\tilde u)}-\frac{\L(x)}{2}-\enorm{s_y(u)-s_y(\tilde u)}+\frac{\L(y)}{2}}\\
&\leq 3c_1\norm{x-y}_\h1
\end{align*}
A similar argument applies to the denominator $\sinh(-\L(x)/2)$ and moreover
inequality \eqref{eq:lengthbound} ensures that $\sinh(-\L(x)/2)$ is bounded
away from zero. Since the quotient of two Lipschitz functions is itself
Lipschitz provided the denominator is bounded away from zero, we have that $G$
is Lipschitz, i.e. there is a constant $c_2>0$ such that
\[
 \enorm{G(x;u,\tilde u)-G(y;u,\tilde u)}\leq c_2\norm{x-y}_\h1
\,.
\]
Now we have
\begin{align*}
F(x)(u) &- F(y)(u)\\
&=-x(u)+y(u)-\int_0^1 x(\tilde u)G(x;u,\tilde u)|x'(\tilde u)|-y(\tilde u)G(y;u,\tilde u)|y'(\tilde u)|\,d\tilde u
\\
&= \begin{multlined}[t]-(x(u)-y(u))
\\
-\int_0^1 (x-y)G(x)|x'|+y G(x)\left (|x'|-|y'| \right )+y|y'|\left (G(x)-G(y)\right)\, d\tilde u
\,.
\end{multlined}
\end{align*}
By \eqref{eq:lengthbound} and \eqref{eq:Gest1} there exists $c_0$ such that $|G(x;u,\tilde u)|\leq c_0$ for all $x\in Q_b$.
Then using the Lipschitz condition for $G$ we find
\begin{equation*} 
\begin{split}
&|F(x)-F(y)|\\
&\leq \begin{multlined}[t] |x-y|\\+\int_0^1 c_0 |x-y||x'|+c_0 |y||x'-y'|+c_2|y||y'|\norm{x-y}_\h1 \, d\tilde u \end{multlined}\\
&\leq \begin{multlined}[t] |x-y|+c_0\norm{x-y}_\l2 \norm{x'}_\l2+c_0 \norm{y}_\l2 \norm{x'-y'}_\l2 \\
 +c_2\norm{y}_\l2 \norm{y'}_\l2 \norm{x-y}_\h1
\end{multlined}
\end{split}
\end{equation*}
Therefore, recalling that $x,y\in Q_b$ satisfy  $\norm{x},\norm{y}\leq \norm{x_0}_1+b$, integrating the above gives
\begin{equation}
	\label{eq:Flip1}  \norm{F(x)-F(y)}_\l2\leq  \text{const} \norm{x-y}_\h1 
\end{equation}
We need a similar result for
\begin{multline}\label{eq:Fdash}
F(x)'(u)-F(y)'(u)=\\-x'(u)+y'(u)- \int_0^1 (x-y)\partial_uG(x)\norm{x'}+y \partial_uG(x)\left (\norm{x'}-\norm{y'} \right )\\+y\norm{y'}\left (\partial_uG(x)-\partial_uG(y)\right ) d\tilde u
\end{multline}
Comparing \eqref{eq:dG}, define 
\[
 A(x;u,\tilde u):=\frac{\sinh\left ( |s_x(u)-s_x(\tilde u)|-\tfrac{\L(x)}{2} \right )}{2 \sinh(-\tfrac{\L(x)}{2})}
\]
so that 
\[\partial_u G(u,\tilde u)=A(x;u,\tilde u)\norm{x'(u)}\sign(u-\tilde u)
\,.
\]
Then as in $\eqref{eq:Gest2}$ we have $|A(x)(u,\tilde u)|\leq \frac{1}{2}$ and arguing as for $G$ above we also have that $A$ is Lipschitz:
\[ |A(x;u,\tilde u)-A(y;u,\tilde u)|\leq \text{const} \norm{x-y}_\h1 \]
Now
\begin{align*}
\int_0^1  \enorm{ \partial_uG(x)-\partial_uG(y)}d\tilde u
&= \int_0^1\enorm{A(x) \vphantom{\frac{}{}} \enorm{x'(u)}-A(y)\enorm{y'(u)}}d\tilde u \\
&=\int_0^1\enorm{\vphantom{\frac{}{}} \left (A(x)-A(y)\right )\enorm{x'(u)}+A(y)\left (\enorm{x'(u)}-\enorm{y'(u)}\right )}d\tilde u\\
&\leq \enorm{x'(u)}\int_0^1\enorm{A(x)-A(y)}\,d\tilde u +\enorm{x'(u)-y'(u)}\int_0^1\enorm{A(y)}d\tilde u\\
&\leq \text{const}\enorm{x'(u)}\norm{x-y}_\h1+\tfrac{1}{2}\enorm{x'(u)-y'(u)}
\,.
\end{align*}
Using this estimate in \eqref{eq:Fdash}, together with \eqref{eq:Gest2} gives
\begin{multline*}
\enorm{F(x)'(u)-F(y)'(u)}\\
\leq \enorm{x'(u)-y'(u)}+\tfrac{1}{2}\enorm{x'(u)}\norm{x-y}_\l2\norm{x'}_\l2+\tfrac{1}{2}\enorm{x'(u)}\norm{y}_\l2\norm{x'-y'}_\l2\\
+ \left( \text{const}\enorm{x'(u)}\norm{x-y}_\h1+\tfrac{1}{2}\enorm{x'(u)-y'(u)}\right )\norm{y}_\l2\norm{y'}_\l2
\end{multline*}
and then
\begin{multline}\label{eq:Flip2}
 \norm{F(x)'-F(y)'}_\l2\\
\begin{split}
& \leq \norm{x-y}_\l2+ \tfrac{1}{2}\norm{x'}_\l2^2\norm{x-y}_\l2 +\tfrac{1}{2}\norm{x'}_\l2\norm{y}_\l2\norm{x'-y'}_\l2  \\
  & \qquad + \text{const}\norm{x'}_\l2 \norm{y}_\l2\norm{y'}_\l2\norm{x-y}_\h1+\tfrac{1}{2}\norm{x'-y'}_\l2\norm{y} _\l2\norm{y'}_\l2 
\\
& \leq \text{const} \norm{x-y}_\h1 
\,.
\end{split}
\end{multline}
Combining \eqref{eq:Flip1} and \eqref{eq:Flip2} gives the required estimate, there exists $L$ such that
\[  \norm{F(x)-F(y)}_\h1\leq L \norm{x-y}_\h1 \]
for all $x,y \in Q_b$.
\end{proof}

\begin{proof}[Proof of Theorem \ref{shortexistence}]
According to the generalised (to Banach space) Picard-Lindel\"of theorem in
\cite{Zeidler:1986aa} (Theorem 3.A), the estimates in Lemma \ref{blip}
guarantee existence and uniqueness of a solution on the interval provided 
$KT_0 < b$.
\end{proof}

\subsubsection{Global existence} \label{ssglobal}


We may extend the existence interval by repeated applications of the
Picard-Lindel\"of theorem from \cite{Zeidler:1986aa}.

There are two issues to be resolved for this.
First, the constants $K$ and $L$ from Lemma \ref{blip} depend on the $H^1$-norm of $x_0$.
When we attempt to continue the solution, we must show that in the forward and
backward time directions this norm does not explode in finite time to
$+\infty$.

Second, the flow must remain within $\mathcal{Q}_b$ for some $b$; as the
evolution continues forward, length is decreasing, and so the amount of time
that we can extend depends not only on the $H^1$-norm of the solution but also
the length bound from below.
In the backward time direction length is in fact \emph{increasing}, so this
second issue does not arise there.

First, we study the $L^\infty$-norm of the solution.

\begin{lemma}
\label{uniformLinf} 
Let $X$ be an $H^1(ds)$ curve shortening flow defined on some interval $(-T,T)$.
Then $\norm{X(t)}_\infty $ is non-increasing on $(-T,T)$.
Furthermore, we have the estimate
\[
	\vn{X(t)}_\infty \le e^{-2t}\vn{X(0)}_\infty\,,\quad\text{for $t<0$}
	\,.
\]
\end{lemma}
\begin{proof}
In the forward time direction, we proceed as follows for the uniform bound.
For any $t_0\in (-T,T)$ there exists $u_0$ such that $\norm{X(t_0)}_\infty =|X(t_0,u_0)| $ and then 
\begin{align*}
\left. \frac{d}{dt} |X(t,u)|^2 \right|_{(t_0,u_0)}
&= 2\ip{X(t_0,u_0), X_t(t_0,u_0)} \\
&= -2|X(t_0,u_0)|^2 - 2\ip{X(t_0,u_0), \int X(t_0,\tilde s)G(s_0, \tilde s) d\, \tilde s}\\
&= -2 \norm{X(t_0)}_\infty^2-2\ip{X(t_0,u_0), \int X(t_0,\tilde s)G(s_0, \tilde s) d\, \tilde s}\\
&\leq -2 \norm{X(t_0)}_\infty^2+2 \norm{X(t_0)}_\infty^2 \int G(s_0, \tilde s) d\, \tilde s\\
&\leq 0.
\end{align*}
Now let $t_1 =\sup \{ t\geq t_0: \norm{X(t)}_{L^\infty}=\enorm{X(t,u_0)} \}$. By the inequality above $\norm{X(t)}_\infty$ is non-increasing for all $t\in [t_0,t_1)$, and by the continuity of $X$ in $t$, $\lim_{t\to t_1} \norm{X(t)}_{L^\infty}=\norm{X(t_1)}_{L^\infty}$. Since $t_0$ was arbitrary, it follows that $\norm{X(t)}_{\infty}$ cannot increase at any $t$.

In the backward time direction, we need an estimate from below.
Let us calculate
\begin{align*}
\frac{d}{dt} \big(e^{t} X(t, u)\big)
	&= e^{t} \Big(
		 (-X(t, u) + X(t, u))
		- \int X(t, \tilde s)G(s,\tilde s)\,d\tilde s
		\Big)
\\
	&= -e^{t} 
		\int X(t, \tilde s)G(s,\tilde s)\,d\tilde s
\end{align*}
so ($u_0$ as before)
\begin{align*}
\frac{d}{dt} \big( e^{2t} |X(t, u)|^2\big) \bigg|_{(t_0,u_0)}
	&= -2e^{2t_0} 
		\ip{X(t_0,u_0), \int X(t_0, \tilde s)G(s_0,\tilde s)\,d\tilde s}
\\
	&\ge -2e^{2t_0} \vn{X(t_0)}_\infty^2
	= -2e^{2t_0} |X(t_0,u_0)|^2
\,.
\end{align*}
Hence 
\[
\frac{d}{dt}\left (e^{4t}|X(t,u)|^2\right )\bigg|_{(t_0,u_0)}\geq 0
\]
and integrating from $t$ to $0$ (assuming $t<0$, and $u_0$ changing as necessary) this translates to
\[
\vn{X(t)}_\infty^2 \le 
e^{-4t}\vn{X(0)}_\infty^2
\]
This is the claimed estimate in the statement of the lemma.
\end{proof}

\begin{lemma} \label{uniformH1flat}
Let $X$ be an $H^1(ds)$ curve shortening flow defined on some interval $(-T,T)$.
Then $\norm{X_u(t)}_\l2$ is non-increasing on $(-T,T)$. 
Furthermore, we have the estimate
\[
	\vn{X_u(t)}_\l2 \le e^{-t}\vn{X_u(0)}_\l2\,,\quad\text{for $t<0$}
	\,.
\]
\end{lemma}
\begin{proof}
As in \eqref{eq:fu} 
\begin{align*}
X_{tu}
&=-X_u-\enorm{X_u(u)}\int_0^\L X_{\tilde s} G\, d\tilde s
\end{align*}
and therefore, recalling that $\int_0^\L |G(s,\tilde s)| d\tilde s=1$, 
\begin{align*}
\frac{d}{dt}\int_0^1 \enorm{X_u}^2 \, du &=2\int_0^1 \ip{X_{ut}, X_u} \, du
\\
&=-2\int_0^1 \ip{X_u, X_u} \, du -2\int_0^1 \ip{X_u, \enorm{X_u}\int_0^\L X_{\tilde s} G\, d\tilde s\,} du
\\
&\leq  -2\norm{X_u}_{L^2}^2+ 2\int_0^1 \enorm{X_u}^2\int_0^\L |G| \, d\tilde s \, du
\\
&\leq 0
\,.
\end{align*}
This settles the forward time estimate.
As before, for the backward time estimate we need a lower bound.
We calculate
\begin{align*}
\frac{d}{dt}\bigg(e^{2t}\int_0^1 \enorm{X_u}^2 \, du\bigg)
&=2e^{2t}\int_0^1 \ip{X_u, \enorm{X_u}\int_0^\L X_{\tilde s} G\, d\tilde s}\, du
\\
&\ge  -2e^{2t}\norm{X_u}_{L^2}^2
\,.
\end{align*}
The same integration as in the backward time estimate for Lemma
\ref{uniformLinf} yields the claimed backward in time estimate.
\end{proof}

The estimates of Lemmata \ref{uniformLinf}, \ref{uniformH1flat} yield the following control on the $H^1$-norm of the solution.

\begin{corollary}
\label{apriori}
Let $X$ be an $H^1(ds)$ curve shortening flow defined on some interval $(-T,T)$. 
Then
\[
 \norm{X(t)}_\h1^2 \leq \norm{X(0)}_\h1^2\,,\qquad \text{for all } t\in [0,T)
\,,
\]
and
\[
 \norm{X(t)}_\h1^2 \leq \norm{X(0)}_\h1^2e^{-4t}\,,\qquad \text{for all } t\in (-T,0)
\,.
\]

\end{corollary}

A similar technique allows us to show also that if the initial data for the flow is an immersion, it remains an immersion.

\begin{lemma}\label{immersion}
Let $X$ be an $H^1(ds)$ curve shortening flow defined on some interval $(-T,T)$ with $X(0)\in \imm^1$.
Then $X(t)\in \IM$ for all $t\in (-T,T)$.
\end{lemma}
\begin{proof}
Using \eqref{eq:fu} we have
\begin{equation}\label{eq:xuevo}
\frac{d}{dt}|X_u|^2=2\langle X_{ut},X_u \rangle = -2|X_u|^2-2|X_u|\langle X_u,\int_0^\L X_{\tilde s} G \, d\tilde s\rangle 	
\,.
\end{equation} 
Now since $\int_0^\L X_s \,ds=0$ we have
\begin{align*}
	\int_0^\L X_s(\tilde s) G \, ds&=\int_0^\L X_s(\tilde s)G(s,\tilde s) d\tilde s-G(s,s)\int_0^\L X_s(\tilde s)\,d\tilde s \\
&=\int_0^\L X_s (G(s,\tilde s)-G(s,s)) \, d\tilde s \\
&=\int_0^\L X_s \int_s^{\tilde s} G_{\tau}(s,\tau)d\tau d\tilde s
\end{align*}
so using $|G_s|\leq \frac{1}{2}$ (cf. \eqref{eq:Gest2}) we find
\begin{equation}\label{eq:TGbound}
\bigg|\int_0^\L X_s(\tilde s) G \, ds\bigg|
\leq 
	\int \frac{1}{2}|s-\tilde s| \, d\tilde s \leq \L^2/2 
\,.
\end{equation}

Using this estimate with \eqref{eq:xuevo} yields
\begin{align}\label{eq:xuineq}
	(-2- \L^2)\enorm{X_u} &\leq \frac{d}{dt}\enorm{X_u}^2\leq (-2+\L^2)\enorm{X_u}^2
\,.
\end{align}
From $\frac{d\L}{dt}=-\norm{\grad \L_X}^2_{H^1(ds)}$ we know that $\L$ is non-increasing and so rearranging the inequality on the left and multiplying by an exponential factor gives
\[	0\leq \frac{d}{dt}\left (e^{(2+\L(0)^2)t}\enorm{X_u}^2\right ) \, .\]
Now integrating from $0$ to $t>0$ gives
\begin{equation}\label{eq:ximm}
	|X_u(0,u)|^2 e^{-(2+\L(0)^2)t}\leq |X_u(t,u)|^2 
\end{equation}
Then since $X_u$ is initially an immersion, it remains so for $t>0$.
The estimate backward in time is analogous but we instead use the second inequality in \eqref{eq:xuineq} and integrate from $t<0$ to $0$.
The statement is
\begin{equation}\label{eq:ximmback}
	|X_u(0,u)|^2 e^{-(\L(0)^2-2)t}\leq |X_u(t,u)|^2 
\end{equation}
for $t<0$.

\end{proof}

\begin{corollary}\label{lengthbelow}
Let $X$ be an $H^1(ds)$ curve shortening flow defined on some interval $(-T,T)$ with $T<\infty$.
Then there exists $\varepsilon>0$ such that $\L(X(t))>\varepsilon$ for all $t\in (-T,T)$.
\end{corollary}
\begin{proof}
For $t\ge 0$, taking the square root in \eqref{eq:ximm} and then integrating over $u$ gives
\[
 \L(0)e^{-(1+\L(0)^2/2)t} \leq \L(t)
\,.
\]
Since $\L(t)$ is non-increasing the result follows.
\end{proof}

\begin{theorem}
\label{TMglob}
Any $H^1(ds)$ curve shortening flow defined on some interval $(-T,T)$ with $T<\infty$ may be extended to all $t\in (-\infty,\infty)$. 
\end{theorem}
\begin{proof}
According to Lemma \ref{blip} and Theorem \ref{shortexistence}, given $x_0\in
H^1(\S^1,\R^2)$, for $0<\varepsilon_0<\frac{\L(x_0)}{K}$ there is a unique
solution $X(t,u)$ for $t\in [t_0-\varepsilon_0,t_0+\varepsilon_0]$.
Following the standard continuation procedure, one takes
$x_{\pm1}:=X(t_0\pm\varepsilon_0)$ as the initial condition for a new application of
Theorem \ref{shortexistence} with existence time $\varepsilon_1<\L(x_1)/K_1$, and so on.

Take $\overline{T}$ to be the maximal time such that the flow $X$ can be extended forward: $t\in (-T, \overline{T})$.
If $\overline{T} <\infty$, then one or more of the following have occurred:
\begin{itemize}
\item $\L(X(t))\searrow0$ as $t\nearrow\overline{T}$;
\item $\vn{X(t)}_{H^1}\rightarrow\infty$ as $t\nearrow\overline{T}$.
\end{itemize}
The first possibility is excluded by Corollary \ref{lengthbelow}, and the second is excluded by Corollary \ref{apriori} (we use $\overline{T}<\infty$ here).
This is a contradiction, so we must have that $\overline{T} = \infty$.

The argument in the backward time direction is completely analogous: suppose 
$\overline{T}$ is the maximal time such that the flow $X$ can be extended backward: $t\in (\overline{T}, T)$.
If $\overline{T} > -\infty$, then one or more of the following have occurred:
\begin{itemize}
\item $\L(X(t))\rightarrow0$ as $t\searrow\overline{T}$;
\item $\vn{X(t)}_{H^1}\rightarrow\infty$ as $t\searrow\overline{T}$.
\end{itemize}
The first possibility is excluded by the fact that the flow decreases length.
The second is excluded by Corollary \ref{apriori} (we again use $\overline{T}>-\infty$ here).
This is a contradiction, so we must have that $\overline{T} = -\infty$.
\end{proof}

\subsection{Convergence}\label{ssconvergence} 

In this subsection we examine the forward in time limit for the flow.
The backward limit is not expected to have nice properties.
One way to see this is in the $H^1(ds)$ length of the tail of an $H^1(ds)$
curve shortening flow.
(We will see in Lemma \ref{h1dslength} that the $H^1(ds)$ length of any forward
trajectory is finite.)
For instance, a circle evolving under the flow has radius $r(t) =
\sqrt{W(e^{c-2t})}$.
The $H^1(ds)$ length is larger than the $L^2(ds)$ length, and this grows (as
$t\rightarrow-\infty$) linear in $W(e^{c-2t})$.
This is not bounded, and so in particular the $H^1(ds)$ length of any negative
tail is unbounded.

Throughout we let $X:(-\infty,\infty)\to H^1(\S^1,\R^2)$ be a solution to the $H^1(ds)$ curve shortening flow \eqref{eq:h1csf}.
We will prove $\lim_{t\to \infty} X(t)$ exists and is equal to a constant map. In order to use the $H^1(ds)$ gradient (see Remark  \ref{flowspace})  we present the proof for the case where $X(t)$ is immersed, but the results can be extended to the flow in $H^1(\S^1,\R^2)\setminus \mathcal{C}$ as described in Remark \ref{extension} below.

It will be convenient to define $K$ by 
\[
K(x):=\frac{\cosh (|x|-\L/2)}{2\sinh(\L/2)}, \quad x\in [0,\L]\,,
\]
so that $G=-K(|s-\tilde s|)$,  $K(x)>0$ and  $\int_0^\L K(x)\,dx = 1$.
We take the periodic extension of $K$ to all of $\R$, which we still denote by $K$, and then
\begin{equation}\label{eq:intK}
	\int_0^\L K(s-\tilde s) d s=\int_{-\tilde s}^{\L-\tilde s}K(x)dx=1
\end{equation}
Define, for any $\gamma\in H^1(\S^1,\R^2)$, 
\[
 (\gamma \ast K)(s) :=\int_{\S^1} \gamma (\tilde s)K(s-\tilde s) \, d\tilde s \,.
\]
This is a so-called `nonlinear' (in $\gamma$) convolution.
We have the following version of Young's convolution inequality.
\begin{lemma}
For any $\gamma\in H^1(\S^1,\R^2)$ and $K$,$*$ as above, we have
\begin{align}
\norm{\gamma \ast K }_{L^2(ds)}& \leq	 \norm{ \gamma}_{L^2(ds)} \label{eq:convolution} 
\,.
\end{align}
\end{lemma}
\begin{proof}
Write
\[ \enorm{\gamma(\tilde s)}|K(s-\tilde s)|=\left (\enorm{\gamma(\tilde s)}^2|K(s-\tilde s)|\right )^{\frac{1}{2}}|K(s-\tilde s)|^{\frac{1}{2}} \]
then by the H\"older inequality and \eqref{eq:intK}
\begin{equation*}
	\int_0^\L \enorm{\gamma(\tilde s)}|K(s-\tilde s)|\, d\tilde s \leq \left ( \int \enorm{\gamma(\tilde s)}^2|K(s-\tilde s)| \,d\tilde s \right)^{\frac{1}{2}}
\,.
\end{equation*}
Hence 
\begin{align*}
	\norm{\gamma \ast K }^2_{L^2(ds)}& \leq  \int \left (\int \enorm{\gamma(\tilde s)}|K(s-\tilde s)| \,d\tilde s \right )^2 ds \\ 
&\leq \int \int \enorm{\gamma(\tilde s)}^2|K(s-\tilde s)| \,d\tilde s ds\\ 
&\leq \int\enorm{\gamma(\tilde s)}^2\int |K(s-\tilde s)|\, ds\,d\tilde s\\
&\leq \norm{ \gamma}^2_{L^2(ds)} 
\,.
\end{align*}
\end{proof}

The convolution inequality implies the following a-priori estimate in $L^2$.

\begin{lemma} \label{l2apriori}
Let $X$ be an $H^1(ds)$ curve shortening flow.
Then $\norm {X(t)}_{L^2(ds)} $ is non-increasing as a function of $t$. 
\end{lemma}
\begin{proof}
First note that since $G_s=-G_{\tilde{s}}$ we have
\[
X_{ts}=-X_s-\int_0^\L  X G_s \,d\tilde s= -X_s-\int_0^\L X_s(\tilde s) G \, ds\,.
\] 
Then (using $\frac{d}{dt} ds=\ip{X_{ts}, X_s}\, ds$) we find  
\begin{align*}
\frac{d}{dt}\norm{X(t)}^2_{L^2(ds)}& =2\int_0^\L \ip{X_t, X} \,ds +\int_0^\L \enorm{X}^2 \ip{X_{ts}, X_s }ds \\
&= \begin{multlined}[t]
-2 \int_0^\L\enorm{X}^2 \, ds- 2\int_0^\L \ip{X(s), \int_0^\L X(\tilde s) G\, d\tilde s }\,ds \\
-\int_0^\L \enorm{X}^2 \, ds -\int_0^\L \enorm{X}^2 \ip{X_s, \int_0^\L X_s(\tilde s) G \, d\tilde s }\, ds 
\,.
 \end{multlined}
\end{align*}
H\"older's inequality and the convolution inequality \eqref{eq:convolution} now yield
\begin{align*}
\frac{d}{dt}\norm{X(t)}^2_{L^2(ds)} &\leq 
 -3 \int_0^\L\enorm{X}^2 \, ds + 2\int_0^\L \enorm{X}\enorm {X\ast K} \,ds 
+ \int \enorm{X}^2 \int_0^\L |G| \, d\tilde s \, ds
\\
&\leq 0
\,.
\end{align*}
\end{proof}

Now we give a fundamental estimate for the $H^1(ds)$-gradient of length along the flow.

\begin{lemma} \label{lem:gradientinequality}
Let $X$ be an $H^1(ds)$ curve shortening flow.
There exists a constant $C>0$ depending on $X(0)$ such that
\begin{equation} \label{eq:gradientinequality}
	\norm{\grad_{H^1(ds)}\L_{X(t)}}_{H^1(ds)}\geq C \L(X(t))^{\frac{1}{2}}
\end{equation} 
for all $t\in [0,\infty)$.
\end{lemma}
\begin{proof}
From \eqref{eq:dlength} and \eqref{eq:h1grad}, if $X$ is a solution of \eqref{eq:h1csf} then
\begin{align*}
\frac{d}{dt}\L(X)&=d\L_X(-\grad_{H^1(ds)}\L_{X})\\
& =\int_0^\L \langle X_{ss}(t,s),X(t,s)+\int_0^\L X(t,\tilde s)G(s,\tilde s)d\tilde s \rangle d s
\,.
\end{align*}
Integration by parts with \eqref{eq:fundamental} gives
\begin{align}\nonumber 
\frac{d}{dt}\L(X)
&=-\int_0^\L \langle X_s,X_s\rangle ds+\int_0^\L \langle X,X\rangle ds\\ \nonumber
&\qquad \qquad +\int_0^\L\int_0^\L \langle X(t,s),X(t,\tilde s)G(s,\tilde s)\rangle d\tilde s ds \\ 
&= -\L(X)-\int_0^\L \langle X,X_t\rangle ds
\,.
 \label{eq:lengthevo}
\end{align}
Since $\frac{d\L}{dt}=-\Vert \grad_{H^1(ds)}\L_{X}\Vert _{H^1(ds)}^2$, we have
\begin{align*}
	\norm{\grad_{H^1(ds)}\L_{X}}_{H^1(ds)}^2&=\L -\int_0^\L \ip{X, \grad_{H^1(ds)}\L_{X}} \, ds \\
&\geq \L -\norm{X}_{L^2(ds)}\norm{\grad_{H^1(ds)}\L_{X}}_{L^2(ds)}
\,.
\end{align*}
The inequality $2ab\leq \varepsilon a^2 +\frac{1}{\varepsilon}b^2$ for all $\varepsilon>0$ implies
\begin{align*}
	\L & \leq \frac{\varepsilon}{2} \norm{X}_{L^2(ds)}^2+ \left ( \frac{1}{2\varepsilon}+1\right )\norm{\grad_{H^1(ds)}\L_{X}}^2_{H^1(ds)}\\
&\leq \frac{\varepsilon}{2}  \norm{X}_{\infty}^2 \L+ \left ( \frac{1}{2\varepsilon}+1\right )\norm{\grad_{H^1(ds)}\L_{X}}^2_{H^1(ds)}
\,.
\end{align*}
Now Lemma \ref{uniformLinf} yields
\[ \L \left ( 1-\frac{\varepsilon}{2} \norm{ X(0)}_{\infty}^2 \right) \leq \left ( \frac{1}{2\varepsilon}+1\right )\norm{\grad_{H^1(ds)}\L_{X}}^2_{H^1(ds)} \]
and choosing $\varepsilon$ sufficiently small gives \eqref{eq:gradientinequality}.
\end{proof}

The gradient inequality immediately implies exponential decay of length.

\begin{lemma}
\label{LMexpdeclen}
Let $X$ be an $H^1(ds)$ curve shortening flow.
The length $\L(X(t))$ converges to zero exponentially fast as $t\to \infty$.
\end{lemma}
\begin{proof}
Using the gradient inequality \eqref{eq:gradientinequality} we have
\begin{align*}
	-\frac{d \L}{dt}=\norm{\grad_{H^1(ds)}\L_{X}}^2_{H^1(ds)} \geq C\L
\,.
\end{align*}
Integrating gives
\begin{equation}\label{eq:Lexp}
	\L(t)\leq \L(0) e^{-Ct}
\end{equation}
as required.
\end{proof}

Another consequence of the gradient inequality is boundedness of the $H^1(ds)$ length of the positive trajectory $X$.

\begin{lemma} \label{h1dslength}
Let $X$ be an $H^1(ds)$ curve shortening flow.
The $H^1(ds)$-length of $\{X(\cdot,t)\,:\,t\in(0,\infty)\}\subset H^1(\S^1,\R^2)$ is finite. 
\end{lemma}
\begin{proof}
From the gradient inequality \eqref{eq:gradientinequality}
\begin{align*}
\frac{d}{dt}\L&=-\norm{ \grad_{H^1(ds)}\L_{X}} _{H^1(ds)}^2
\leq -\norm{\grad_{H^1(ds)}\L_{X}}_{H^1(ds)}	\norm{X_t}_{H^1(ds)}\\
&\leq -C\L^{\frac{1}{2}} \norm{X_t}_{H^1(ds)}
\end{align*}
i.e. 
\[  \frac{d}{dt}(2\L^{\frac{1}{2}})\leq - C\norm{X_t}_{H^1(ds)}
\]
and therefore 
\begin{align}\nonumber
	2\L(t)^{\frac{1}{2}}-2\L(0)^{\frac{1}{2}} &\leq -C \int_0^t \norm{X_t}_{H^1(ds)} \, dt\,,\text{ or}\\ \label{eq:gradinfinity}
\int_0^t \norm{X_t}_{H^1(ds)} \, dt &\leq \frac{2\L(0)^{\frac{1}{2} }}{C} 
\end{align}
Taking the limit $t\to \infty$ in the above inequality, the left hand side  is the length of the trajectory $X:[0,\infty)\to H^1(\S^1,\R^2)$ measured in the $H^1(ds)$ metric. 
\end{proof}

Now we conclude convergence to a point.

\begin{theorem}\label{convergence}
Let $X$ be an $H^1(ds)$ curve shortening flow.
Then $X$ converges as $t\rightarrow\infty$ in $H^1$ to a constant map $X_\infty\in H^1(\S^1,\R^2)$.
\end{theorem}
\begin{proof}
Recalling \eqref{eq:gradbound0} we have
$\enorm{X_t} \leq \L $
and therefore 
\begin{align} \nonumber 
\norm{X_t}_{H^1}^2&=\int \enorm{X_t}^2 du+\int \enorm{X_{tu}}^2 du \\ \label{eq:h1est}
&\leq \L^2 +\int \enorm{X_{tu}}^2 du
\,.
\end{align}
Using $G_s=-G_{\tilde{s}}$ we have
\[ X_{ts}=-X_s-\int_0^\L  X G_s \,d\tilde s= -X_s-\int_0^\L X_s(\tilde s) G \, ds\]
and then from \eqref{eq:TGbound}
\begin{align}\label{eq:normgradLs}
	\enorm{X_{ts}}\leq 1+
			\enorm{ \int_0^\L X_s(\tilde s) G \, ds }
	\leq 1+\frac{\L(X)^2}{2}
\,.
\end{align} 
Hence $\enorm{X_{tu}}\leq \enorm{X_u}(1+\L^2/2)$. Recalling \eqref{eq:xuineq} and then \eqref{eq:Lexp} 
\begin{align}\label{eq:ddtxu}
(-2-\L(0)^2e^{-2Ct})\enorm{X_u}^2&\leq \frac{d}{dt}\enorm{X_u}^2\leq (-2+\L(0)^2e^{-2Ct})\enorm{X_u}
\end{align}
Using the second inequality, multiply by the integrating factor $e^{p(t)}$ where \[p(t):=\int_0^t 2-\L^2(0)e^{-2C\tau} d\tau,\]
and integrate with respect to $t$ to find
\begin{align*}
\enorm{X_u(t)}^2\leq \enorm{X_u(0)}^2e^{-p(t)}\leq \enorm{X_u(0)}^2e^{-2t+c_3}
\end{align*}
for some constant $c_3$. For future reference we note that the same procedure can be applied to the lower bound in \eqref{eq:ddtxu} and then
\begin{equation}\label{eq:Xuest}
\enorm{X_u(0)}^2e^{-2t-c_3}\leq \enorm{X_u(t)}^2 \leq \enorm{X_u(0)}^2e^{-2t+c_3}\,.
\end{equation}
We therefore have
\[ 
\enorm{X_{tu}}\leq \enorm{X_u(0)}e^{-t+c}(1+\L^2/2)
\]
and then referring back to \eqref{eq:h1est}:
\begin{align*}
	\norm{X_t}_{H^1} \leq \L +\norm{X_u(0)}_{L^2}e^{-t+c}(1+\L^2/2)
\,.
\end{align*}
Using the gradient inequality \eqref{eq:gradientinequality} and monotonicity of $\L$ we obtain
\begin{align*}
\norm{X_t}_{H^1} \leq c_1 \norm{X_t}_{H^1(ds)}^2+c_2e^{-t+c}
\,.
\end{align*}
By integrating $\frac{d\L}{dt}=-\norm{X_t}_{H^1(ds)}^2$ with respect to $t$ we have
\[
\int_0^\infty \norm{X_t}_{H^1(ds)}^2 dt \leq \L(0)
\,.
\]
Hence for all $\varepsilon>0$ there exists $t_\varepsilon$ such that $\int_t^\infty \norm{X_t}_{H^1} dt <\varepsilon $ for all $t\geq t_\varepsilon$,
and since 
\[
 \bigg\lVert
	 X(t_2)-X(t_1)
 \bigg\rVert_{H^1}
=
 \bigg\lVert
 \int_{t_1}^{t_2} X_t \, dt
 \bigg\rVert_{H^1}
 \leq
 \int_{t_1}^{t_2}\norm{X_t}_{H^1} dt
\]
it follows that $X_t$ converges in $H^1$ to some $X_\infty$. By \eqref{eq:Lexp} the length of $X_\infty$ is zero, i.e. it is a constant map. 
\end{proof}

\begin{remark}
	If $(\imm^1,H^1(ds))$ were a complete metric space then Lemma
\ref{h1dslength} would be enough to conclude convergence of the flow. However it is shown in \cite{Michor:2007aa} section 6.1 that the $H^1(ds)$ geodesic of concentric circles can shrink to a point in finite time, so the space is not even geodesically complete.
Indeed, Theorem \ref{convergence} demonstrates convergence of the flow with finite path length to a point
outside $\IM$, proving again that $(\imm^1, H^1(ds))$ is not metrically complete. 
\end{remark}

\begin{remark}\label{extension}
For extending the convergence result above to the case of initial data $X(0)\in H^1(\S^1,\R^2)\setminus \mathcal{C}$ the main difficulty is that the function $F(X(u,t))$ in \eqref{eq:ode} is no longer the $H^1(ds)$ gradient (cf. Remark \ref{flowspace}) and so we need some other way of arriving at, for example, equation \eqref{eq:lengthevo}. To do this we can approximate by $C^2$ immersions, as it follows from eg. Theorem 2.12 in \cite{Hirsch:1994aa} that these are dense in $H^1(\S^1,\R^2)$. Given $X(t_0)\in H^1(\S^1,\R^2)\setminus \mathcal{C}$ we let $X_\varepsilon(t_0)$ be an immersion such that  $\norm{X(t)-X_\varepsilon(t)}_{H^1}\leq \varepsilon$ for $t$ in a neighbourhood of $t_0$. Then following \eqref{eq:lengthevo} we have 
\[ \frac{d\L (X)}{dt}:= \lim_{\varepsilon\to 0} \left (-\L(X_\varepsilon)-\int_0^\L \langle X_\varepsilon,(X_\varepsilon)_t\rangle ds \right ) \, . \]
The limit exists because all the terms are bounded by $\norm{X_\varepsilon}_{H^1}$ (for $(X_\varepsilon)_t$ this follows from \eqref{fboundh1}). Similarly $\frac{d\L(X)}{dt}=-\norm{F(X)}^2_{H^1}$ and we proceed with the rest of the proofs by writing $F(X)$ or $X_t$ in place of $-\grad_{H^1(ds)}\L_X$.

\end{remark}



\section{Shape evolution and asymptotics}

\subsection{Generic qualitative behaviour of the flow}

Computational experiments indicate that the flow tends to reshape the initial
data, gradually rounding out corners and improving regularity.
However the scale dependence of the flow introduces an interesting effect: 
when the length becomes small, the `reshaping power' seems to run out and curves shrink approximately self-similarly, preserving regions of low regularity.
This means that corners of small polygons persist whereas corners of large polygons round off under the flow (cf. Figure \ref{fig:squares}).

Heuristically, this is because of the behaviour of $G$ as $\L\to 0$. If we
Taylor expand 
\[ G(s,\tilde s)\approx -\frac{1}{2\sinh(\L/2)} \left (1+\frac{1}{2!}\left (\enorm{s-\tilde s}-\L/2 \right )^2+\ldots \right ) \]
since $\enorm{s-\tilde s}\leq \L$, the constant term dominates when $\L$ is small. Then
\[ X_t\approx -X+\int_0^\L \frac{X}{2\sinh(\L/2)}ds = -X+\frac{\L}{2\sinh(\L/2)}\bar X\]
and $\lim_{\L\to 0}\frac{\L}{2\sinh(\L/2)}=1$ so each point on the curve moves toward its centre.


\subsection{Remarks on the numerical simulations}\label{ssnumerics} The numerical simulations were carried out in Julia using a basic forward Euler method. Curves are approximated by polygons. For initial data we
take an ordered list of vertices $X_i$ in $\R^2$ of a polygon and the length of
$X$ is of course just the perimeter of the polygon. The arc length $s_i$ at
$X_i$ is the sum of distances between vertices up to $X_i$ and for the
arc-length element $ds_i$ we use the average of the distance to the previous
vertex and the distance to the next vertex. The Green's function is then
calculated at each pair of vertices: 
\[ G_{ij}(X)=-\frac{\cosh(|s_i-s_j|-L(X)/2)}{2\sinh(L(X)/2)} \]
and the flow velocity $V_i$ at $X_i$ is
\[ V_i=-X_i-\sum_j X_jG_{ij}(X) ds_j \]
The new position $\tilde X_i$ of the vertex $X_i$ is calculated by forward-Euler with timestep $h$:
$ \tilde X_i=X_i+h V_i$. No efforts were made to quantify errors or test accuracy, but the results appear reasonable and stable provided time steps are not too large and there are sufficiently many vertices. A Jupyter notebook containing the code is available online \cite{code}.

\subsection{Evolution and convergence of an exponential rescaling}

\begin{definition}
Let $X:[0,\infty)\to H^1(\S^1,\R^2)\setminus \SC $ be a solution to the $H^1(ds)$ curve shortening flow \eqref{eq:h1csf}. We define the \emph{asymptotic profile} $Y$ of $X$ as 
\[
 Y(t,u):=e^{t}\left(X(t,u)-X(t,0)\right )\,.
\]
\end{definition}

We anchor the asymptotic profile so that $Y(t,0) = 0$ for all $t$.
This is not only for convenience; if the final point that the flow converges to
is not the origin, then an unanchored profile $\tilde Y = e^t X$ would simply
disappear at infinity and not converge to anything.

The aim in this section is to prove that the asymptotic profile converges.
Simulations indicate that there are a variety of possible
shapes for the limit (once we know it exists); numerically, even a simple
rescaling of the given initial data may alter the asymptotic profile. As in the previous section we present the results under the assumption that $X$ is a flow of immersed curves, but they can be extended to $H^1(\S^1,\R^2)\setminus \mathcal{C}$ by the method described in Remark \ref{extension}. 

We will need the following refinement of the gradient inequality. 

\begin{lemma}\label{alpha}
Let $X$ be an $H^1(ds)$ curve shortening flow.
For any $\alpha \in (0,1)$ there exists $t_\alpha$ such that 
\[
\norm{\grad_{H^1(ds)}\L_{X}}^2_{H^1(ds)}\geq \alpha \L(X) 
\]
for all $t\geq t_\alpha $.
\end{lemma}
\begin{proof}
We abbreviate the gradient to $\grad \L_X$ in order to lighten the notation. Equation \eqref{eq:gradbound0} implies
\begin{equation}\label{eq:gradboundL1}
	\norm{\grad \L_X}_{L^1(ds)}\leq \L(X)^2
\end{equation}
 and therefore from \eqref{eq:lengthevo} and $\frac{d\L}{dt}=-\Vert \grad \L (X)\Vert _{H^1(ds)}^2$:
\begin{align*}
	\norm{\grad \L_X}_{H^1(ds)}^2&=\L -\int_0^\L \langle X, \grad \L_X \rangle  \, ds \\
&\geq \L-\norm{X}_\infty \norm{\grad \L_X}_{L^1(ds)} \\
&\geq \L-\norm{X(0)}_\infty \L^2
\end{align*}
where we have also used Lemma \ref{uniformLinf}. Now using \eqref{eq:Lexp}
\[
 \norm{\grad \L_X}_{H^1(ds)}^2 \geq (1- \norm{X(0)}_\infty \L(0)  e^{-Ct}) \L(t)
\,.
\]
If $\alpha\geq 1-\norm{X(0)}_\infty \L(0) $ we can find the required $t_\alpha$ by solving $\alpha=1- \norm{X(0)}_\infty \L(0)  e^{-Ct_\alpha}$, otherwise $t_\alpha=0$. 
\end{proof}

We also need an upper bound for the gradient in terms of length. 
\begin{lemma}
For $X\in H^1(\S^1,\R^2)$,
\begin{equation} \label{eq:gradh1dsbound}
	\norm{\grad \L_X}_{H^1(ds)}^2\leq \L(X)+2\L(X)^3+\frac{\L(X)^5}{4}
\,.
\end{equation} 
\end{lemma}
\begin{proof}
From \eqref{eq:gradbound0} we have
\begin{align*}
	\norm{\grad \L_X}_{L^2(ds)}^2 \leq \L^3
\,.
\end{align*}
Then \eqref{eq:normgradLs} implies
\[
\norm{(\grad \L_X)_s}^2_{L^2(ds)}\leq \int \left (1+\frac{\L^2}{2} \right )^2 \, ds \leq \L + \L^3 + \frac{\L^5}{4}
\]
and the result follows.
\end{proof}

We now prove convergence of the asymptotic profile along a subsequence of times
-- sometimes this is called \emph{subconvergence}.

\begin{theorem}
\label{TMprofile}
Let $X$ be an $H^1(ds)$ curve shortening flow and $Y$ its asymptotic profile.
There is a non-trivial $Y_\infty\in C^0(\S^1,\R^2)$ such that $Y(t)$ has a
convergent subsequence $Y(t_i)\to Y_\infty$ in $C^0$ as $i\to\infty$.
\end{theorem}
\begin{proof}
We will show that $Y(t)$ is eventually uniformly bounded in $H^1$. 
	First we claim that there exist constants $c_0,c_1>0$ and $t_0<\infty$ such that 
\begin{equation}\label{eq:Ylength}
 c_0<\L(Y(t))<c_1	 \quad \text{for all} \quad t>t_0.
\end{equation}
For the upper bound, from $\L(Y)=e^t\L(X)$, \eqref{eq:lengthevo} and \eqref{eq:gradboundL1}
\begin{align}\nonumber
	\frac{d}{dt}\L(Y)&=e^t\L(X)+e^t\frac{d}{dt}\L(X) =e^t\int_0^{\L(X)}\langle X,\grad \L_X\rangle \, ds\\
&\leq \norm{X(t)}_\infty e^t \L(X)^2 
 \leq \norm{X(0)}_\infty e^t \L(X)^2 \label{eq:lengthY}
\end{align}
From Lemma \ref{alpha}, for any $\alpha\in (0,1)$ there exists $t_\alpha$ such that
\[ \frac{d}{dt}\L(X)\leq -\alpha \L(X), \quad t\geq t_\alpha \]
hence 
$
\L(X(t))\leq \L(X(t_\alpha))e^{-\alpha t}$ for $ t>t_\alpha.	$
 Using this in \eqref{eq:lengthY} with eg. $\alpha=\frac{3}{4}$, 
\[ 
\frac{d}{dt}\L(Y)\leq ce^{-\frac{t}{2}}, \quad t\geq t_{3/4}
\,.
\]
where $c$ is a constant depending on $X(0)$ and $\L(X(t_{3/4}))$. Integrating from $t_{3/4}$ to $t$ gives 
\begin{equation}\label{LYupper}
	\L(Y(t))\leq \L(Y(t_{3/4}))+2ce^{-t_{3/4}/2}-2ce^{-t/2}
\end{equation}
which gives an upper bound for $\L(Y(t))$ for $t\geq t_{3/4}$.
For the lower bound the estimate \eqref{eq:gradh1dsbound} gives 
\[
 -\frac{d}{dt}\L(X) \leq \L(X)+2\L(X)^3+\frac{\L(X)^5}{4} 
\,.
\]
Let $t_\beta$ be such that $\L(X(t))<1$ for all $t>t_\beta$.
(From \eqref{eq:Lexp} we can find $t_\beta$ by solving $1 = \L(0)e^{-Ct_\beta}$.)

Then also using the gradient inequality \eqref{eq:gradientinequality} there is a constant $c$ such that
\[
 \frac{d}{dt}\L(X)\geq  -\L(X)-c\L(X)\norm{\grad \L_X}^2_{H^1(ds)} \qquad t>t_\beta
\,.
\]
Recalling \eqref{eq:gradh1dsbound},
this implies ($t>\max\{t_\beta,t_{3/4}\}$)
\[
 \frac{d}{dt}(e^t\L(X)) \geq  -ce^t\L(X)\norm{\grad \L_X}^2_{H^1(ds)}
 \ge -\hat ce^t\L^2(X) \ge \tilde ce^{-\frac12t}
\,.
\]
Integrating with respect to $t$,  there is a constant 
$c_0$ such that 
\[
 \L(X)\geq c_0e^{-t},  \qquad t>\max\{t_\beta,t_{3/4}\}
\]
and therefore $\L(Y)\geq c_0$ for all $t>t_\beta$. Choosing $t_0$ to be the greater of $t_{3/4}, t_\beta$, we have established the claim \eqref{eq:Ylength}. We claim also that 
\begin{equation}\label{eq:Ybound}
	\norm{Y}_{L^2}\leq c_1, \qquad t>t_0.
\end{equation}
To see this, note that by the Fundamental Theorem of Calculus followed by the H\"older inequality applied to each component of $Y$:
\[ 
|Y|^2 \le \L\int_0^{\L(Y)} |Y_s|^2 ds_Y = \L^2(Y)
 \]
and so \eqref{eq:Ylength} gives $\norm{Y}_{L^2}\leq c_1 $.

Multiplying \eqref{eq:Xuest} by $e^{2t}$ gives 
\begin{equation}
\enorm{X_u(0,u)}^2e^{-c_3}\leq  \enorm{Y_u(t,u)}^2\leq \enorm{X_u(0,u)}^2e^{c_3}
\,.
\label{EQimmersionprofile}
\end{equation}
We therefore have a uniform bound on $\norm{Y_u}_{L^p}$ for $1\leq p \leq
\infty$ in terms of $\norm{X_u(0)}_{L^p}$. In particular, if $X(0)\in
H^1$ we have a uniform $H^1$ bound for $Y$ and then by
the Arzela-Ascoli theorem there is a sequence $(t_i)$ and a $Y_\infty\in W^{1,\infty}$
such that $Y(t_i)\to Y_\infty$ in $C^0$ (cf. \cite{Leoni:2017aa} Theorems 7.28, 5.37
and the proof of 5.38). 
\end{proof}

This result can be quickly upgraded to full convergence using a powerful decay estimate.

\begin{theorem}
\label{TMprofilefullconv}
Let $X$ be an $H^1(ds)$ curve shortening flow and $Y$ its asymptotic profile.
There is a non-trivial $Y_\infty\in H^1(\S^1,\R^2)$ such that $Y(t)\to Y_\infty$ 
in $C^0$ as $t\to\infty$.
\end{theorem}
\begin{proof}
For the evolution of $Y$ we calculate
\[
	Y_t(t,u) = \int Y(t,\tilde u) (G(X;0,s_X(\tilde u)) - G(X;s_X(u),s_X(\tilde u))) |X_{\tilde u}|\,d\tilde u
\,.
\]
The $\frac12$-Lipschitz property for $G$ (from \eqref{eq:Gest2}) implies that
\[
\big|
(G(X;0,s_X(\tilde u)) - G(X;s_X(u),s_X(\tilde u)))
\big|
\le \frac12 |s_X( u)| \le \frac12\SL(Y(t))e^{-t} \le ce^{-t}
\,,
\]
by the estimate \eqref{LYupper} in Theorem \ref{TMprofile}.
The estimates in the proof of Theorem \ref{TMprofile} include $||Y||_\infty \le c$.
Using these we find
\begin{align*}
|Y_t(t,u)|
&= \bigg|
	\int Y(t,\tilde u) (G(X;0,s_X(\tilde u)) - G(X;s_X(u),s_X(\tilde u))) |X_{\tilde u}|\,d\tilde u
	\bigg|
\\
&\le ce^{-2t}||Y||_\infty \SL(Y(t)) \le ce^{-2t}
\,.
\end{align*}
Exponential decay of the velocity implies  full convergence by a standard argument (a
straightforward modification to $C^0$ of the $C^\infty$ argument in
\cite[Appendix A]{AMWW} for instance).
\end{proof}

The convergence result (Theorem \ref{TMprofilefullconv}) applies in great generality.
If the initial data $X_0$ is better than a generic map in $H^1(\S^1,\R^2)\setminus\SC$,
for instance if it is an immersion, has well-defined curvature, or further
regularity, then this is preserved by the flow.
That claim is proved in the next section (see Theorem \ref{TMglobCk}).
In these cases, we expect the asymptotic profile also enjoys these additional properties.
This is established in the $C^2$ space in the section following that (see Theorem \ref{TMcurvatureprofile}).

\begin{remark}
The asymptotic shape is very difficult to determine, in particular, it is not
clear if there is a closed-form equation that it must satisfy.
As mentioned earlier, we see this in the numerics.
We can also see this in the decay of the flow velocity $Y_t$.
It decays not because the shape has been optimised to a certain point, but simply because sufficient time has passed so that the exponential decay terms take over.
The asymptotic profile of the flow is effectively constrained to a tubular neighbourhood of $Y(0)$.
\end{remark}

\subsection{The $H^1(ds)$-flow in $\imm^k$ spaces}

Observe from \eqref{eq:gradbound0} and \eqref{eq:fu} that if $\gamma\in C^1$ then
$\grad \L_\gamma$ is also $C^1$. We might therefore consider the flow with $\imm^1$
initial data as an ODE on $\imm^1$ (instead of $H^1(\S^1,\R^2)\setminus \SC$). In fact the same is true for $\imm^2$
( and moreover $\imm^k$) as we now demonstrate. 

Assume $x\in C^2$ is an immersion, then 
\begin{equation}\label{eq:guu} G_{uu}=\partial_u(|x_u| G_s)=|x_u|^2 G_{ss} -\langle x_{uu},x_s \rangle G_s \end{equation} 
and using $G_{ss}=G_{\tilde s\tilde s}$ as well as integrating by parts we obtain
\begin{align}\nonumber
	(\grad \L_x)_{uu}&=x_{uu}+\int x G_{uu}\, d\tilde s\\
&=x_{uu}+|x_u|^2\int x_{\tilde s\tilde s}G \, d\tilde s-\langle x_{uu}, x_s\rangle \int x_{\tilde s} G \, d\tilde s \label{eq:gradluu}
\,.
\end{align}
Now from 
\begin{equation} \label{eq:xss}
	x_{ss}=\frac{x_{uu}}{|x_u|^2}-\langle x_{uu},x_u\rangle \frac{x_u}{|x_u|^4}
\end{equation}
we have that $\norm{x_{ss}}_\infty$ is bounded provided $|x_u|$ is bounded
away from zero for all $u$. Assuming this is the case we have furthermore from
\eqref{eq:gradluu} that $|(\grad \L_x)_{uu}|$ is bounded and  $\grad
\L_x \in C^2$. We may therefore consider the flow as an ODE in
$\imm^2$. Short time existence requires a $C^2$ Lipschitz estimate. One
can estimate $\norm{\grad \L_x-\grad \L_y}_{C^1} $ much the same as in
Lemma \ref{blip}. From \eqref{eq:gradluu} and product expansions as in Lemma
\ref{blip}:
\begin{align*}
	|(\grad \L_x)_{uu}-(\grad \L_y)_{uu}| \leq 
\begin{multlined}[t]
|x_{uu}-y_{uu}|+c_1|\,|x_u|^2-|y_u|^2|+c_2|x_{ss}-y_{ss}|
	\\+c_3|G(x)-G(y)|+c_4|x_{uu}-y_{uu}|+c_5|x_s-y_s|
\end{multlined}
\,.
\end{align*}
The result now follows from the Lipschitz estimate for $G$, together with
estimates $|x_s-y_s|\leq c|x_u-y_u|$ and $|x_{ss}-y_{ss}|\leq
c|x_{uu}-y_{uu}|$ which also follow from product expansions using eg
\eqref{eq:xss}. 

It follows from \eqref{eq:Xuest} that if $X(0)$ is $C^1$, then $X(t)$ is $C^1$
for all $t<\infty$, and moreover $|X_u(t)|$ is bounded away from zero for
all $t<\infty$, so we have global existence for the $C^1$ flow. 

Suppose $X(t)$ is $C^2$ for a short time, then from \eqref{eq:gradluu}
\begin{align*}
	\frac{d}{dt}|X_{uu}|^2
&=
\begin{multlined}[t]
-2|X_{uu}|^2-2|X_u|^2 \left\langle X_{uu},\int X_{\tilde s\tilde s} G \, d\tilde s\right\rangle 
	\\
+2 \langle X_{uu},X_s\rangle \left\langle X_{uu},\int X_{\tilde s} G \, d\tilde s \right\rangle
\,.
\end{multlined}
\end{align*}
From \eqref{eq:xss} notice $|X_{ss}|\leq 2|X_{uu}|\,|X_u|^{-2}$ and therefore
\begin{align*}
\frac{d}{dt}|X_{uu}|^2 & \leq 2 |X_{uu}|\left (c \norm{X_{uu}}_\infty+|X_{uu}| \right)
\end{align*} 
where $c=\norm{X_u}_\infty \sup_u |X_u|^{-1}$.
Supposing that at time $t_0$, $\norm{X_{uu}}_\infty$ is attained at $u_0$, it follows that
\[
	\frac{d}{dt}|X_{uu}|(u_0,t_0)  \leq c |X_{uu}|(u_0,t_0) 
\]
and therefore $\norm{X_{uu}(u_0,t)}\leq e^{ct}$.
By the short time existence $\norm{X_{uu}}_\infty$ is continuous in $t$, so in fact 
\[ \norm{X_{uu}}_\infty\leq e^{ct} \]
and we have global $C^2$. 

For the $\imm^k$ case there is little that is novel and much that is tedious. Claim:
\begin{equation}\label{eq:guk}
	\partial_u^k G=\enorm{X_u}^k\partial_s^kG-\langle \partial_u^k X,X_s\rangle G_s+\sum_i^{k-1} P_i(X_u,\ldots ,\partial_u^{k-1} X) \partial_s^i G
\end{equation} 
where each $P_i$ is polynomial in the derivatives of $X$ up to order $k-1$. From \eqref{eq:guu} this is true for $k=2.$ Assuming it is true for $k$ we have
\begin{align*}
	\partial_u^kG&=
\begin{multlined}[t]
\enorm{X_u}^{k+1}\partial_s^{k+1}G+k\enorm{X_u}^{k-2}\langle X_{uu},X_u\rangle \partial_s^k G-\langle \partial_u^{k+1}X,X_s \rangle G_s \\
-\langle \partial_u^kX,X_{ss}\rangle \enorm{X_u} G_s-\langle \partial_u^k X,X_s\rangle \enorm{X_u}G_{ss}+\partial_u\left (\sum_i^{k-1} P_i \partial_s^i G \right) 
\end{multlined}\\
&=\enorm{X_u}^{k+1}\partial_s^{k+1}G-\langle \partial_u^{k+1},X_s\rangle G_s+\sum_i^k \tilde P_i \partial_s^iG
\end{align*}
where each $\tilde P_i$ is polynomial in the derivatives of $X$ up to order $k$. 
From \eqref{eq:guk} and $\partial_s^kG=-\partial_{\tilde s}^kG $ we can calculate
\begin{align}\nonumber
\partial_u^k (\grad \L_X) &=\partial_u^k X+\int X\partial_s^k G \, d\tilde s\\
&=
\begin{multlined}[t]
\partial_u^kX+\enorm{X_u}^k\int \partial_{\tilde s}^k X G\, d\tilde s-\langle \partial_s^k X,X_s\rangle \int X_{\tilde s} G \, d\tilde s\\
 +\sum_i^{k-1}P_i\int \partial_{\tilde s}^i X G\, d\tilde s
\end{multlined} \label{eq:gradluk}
\end{align}
and observe that if $X$ is in $\imm^k$ then so is $\grad \L_X$. We may therefore
consider the gradient flow  as an ODE in $\imm^k$. Short time existence
requires a $C^k$ Lipschitz estimate. We claim that such an estimate can be
proved inductively using \eqref{eq:gradluk} by similar methods to those used
above for the $C^2$ case, except with longer product expansions. As it is the
same technique but only with a longer proof, we omit it.

In summary, we have:

\begin{theorem}
\label{TMglobCk}
Let $k\in\N$ be a natural number.
For each $X_0\in \imm^k$ there exists a unique eternal
$H^1(ds)$ curve shortening flow $X:\S^1\times\R\rightarrow\R^2$ in
$C^1(\R; \imm^k)$ such that $X(\cdot,0) = X_0$.
\end{theorem}


\subsection{Curvature bound for the rescaled flow}

In this subsection, we study the $H^1(ds)$ curve shortening flow in the space of $C^2$ immersions.
This means that the flow has a well-defined notion of scalar curvature.
Note that while the arguments in the previous section show that the $C^2$-norm of $X$ is bounded for all $t$, they do not show that this bound persists through to the limit of the asymptotic profile $Y_\infty$.
They need to be much stronger for that to happen: not only uniform in $t$, but on $X$ they must respect the rescaling factor.

The main result in this section (Theorem \ref{TMcurvatureprofile}) states that
this is possible, and that the limit $Y_\infty$ of the asymptotic profile in
the $C^2$-space enjoys $C^2$ regularity, being an immersion with bounded
curvature.

We start with the commutator of $\partial_s$ and $\partial_t$ along the flow $X(t,u)$. Given a differentiable function $f(u,t)$:
\begin{align}\label{eq:commutator}
 f_{st}=-\frac{\langle X_{ut},X_u\rangle }{\enorm{X_u}^2}f_u+\frac{1}{\enorm{X_u}}f_{ut}=f_{ts}-\langle X_{ts},X_s\rangle f_s
	\,.
\end{align}
From $X_{ss}=kN$ we have $X_{sst}=kN_t+k_tN$ and then using $\langle N_t,N\rangle =0$,
\[
	k_t=\langle X_{sst},N\rangle
\,.
\]
Applying \eqref{eq:commutator} twice
\begin{align*}
	 X_{sst}&=X_{sts} -\langle X_{ts},X_s\rangle X_{ss}\\
&=\left (X_{ts}-\langle X_{ts},X_s\rangle X_s\right )_s -\langle X_{ts},X_s\rangle  X_{ss}\\
&=X_{tss}-\langle X_{tss},T\rangle T-\langle X_{ts},kN\rangle T-2\langle X_{ts},T\rangle kN
\end{align*} 
and then
\begin{align*}
	k_t& =\langle X_{tss},N\rangle -2k\langle X_{ts},T\rangle\\
&=-\langle (\grad \L_X)_{ss},N\rangle +2k\langle (\grad \L_X)_s,T\rangle \\
&= -k-\langle \grad \L_X, N\rangle +2k\langle T+T\ast G,T\rangle \\
&= k-\langle \grad \L_X, N\rangle + 2k\langle T\ast G,T\rangle
\end{align*}
where $(\grad \L_X)_{ss}-\grad \L_X=kN$ (from \eqref{eq:gradODE}) and $\grad \L_s=T+T\ast G$  have been used. Therefore
\begin{equation}\label{eq:dksquared}
	\frac{d}{dt}k^2= 2k^2-2\langle \grad \L_X, kN\rangle +4k^2 \langle T*G, T\rangle 
\end{equation}
Now letting
\[ \varphi(t):= k_Y^2=e^{-2t}k^2 \]
using \eqref{eq:gradbound0} and \eqref{eq:TGbound} to estimate  \eqref{eq:dksquared}, we find
\[
	\varphi'(t)\leq 2e^{-2t}(|k|\L+k^2\L^2)
              \leq 2e^{-2t}+\frac52\L^2\varphi(t)
	      \,.
\]
Note that in the second inequality we used $a \le 1 + a^2/4$, which holds for any $a\in\R$.
Integration gives 
\[ |\varphi(t)|\leq c_1+\int_0^t c\L^2|\varphi| d\tau \] 
and so by the Bellman inequality (\cite{Pachpatte:1998aa} Thm. 1.2.2)
\[ \varphi(t)\leq c e^{\int_0^t \L^2 \, d\tau } \]
Since $\L(X)$ decays exponentially \eqref{eq:Lexp}, we have that $\varphi$ is uniformly bounded.

This gives stronger convergence for $Y$ in the case of $C^2$ data, and we conclude the following. (Note that the fact $Y_\infty$ is an immersion followed already from \eqref{EQimmersionprofile}.)

\begin{theorem}
\label{TMcurvatureprofile}
Let $X$ be an $H^1(ds)$ curve shortening flow with $X(0)\in \imm^2 $, and $Y$ its asymptotic profile.
There is a non-trivial $Y_\infty\in \imm^2$ such that $Y(t)\to Y_\infty$
in $C^0$ as $t\to\infty$. That is, the asymptotic profile converges to a unique
limit that is immersed with well-defined curvature.
\end{theorem}

\subsection{Isoperimetric deficit}

In this section, we show that the isoperimetric deficit of the limit of the
asymptotic profile $Y_\infty$ is bounded in terms of the isoperimetric profile
of $X_0$.
This is in a sense optimal, because of the great variety of limits for the
rescaled flow, it is not reasonable to expect that the profile always improves.
Indeed, numerical evidence suggests that the profile is not monotone under the flow.
Nevertheless, it is reasonable to hope that the flow does not move the
isoperimetric deficit too far from that of the initial curve, and that's what
the main result of this section confirms.

\subsubsection{Area}
We start by deriving the evolution of the signed enclosed area.
Using \eqref{eq:commutator} we find
\[
	X_{st}=X_{ts}-\langle X_{ts},X_s\rangle X_s=X_{ts}-\langle X_{ts},T\rangle T 
	\,.
\]
Differentiating $\langle N,T \rangle =0$ and $\langle N,N\rangle =1$ with respect to $t$ yields
\begin{align*}
	\langle N_t,T\rangle &=-\langle N,X_{st} \rangle\,,\text{ and}\\
\langle N_t, N\rangle &=0\,.
\end{align*}
Therefore 
\begin{equation} \label{eq:normalevo}
N_t=-\langle N,X_{st}\rangle T=-\langle N, X_{ts} \rangle T	
\,.
\end{equation}
Using the area formula $A=-\frac{1}{2}\int_0^\L \langle X, N\rangle ds$, and $ds=\enorm{X_u}du$ implies $\frac{d}{dt}ds=\langle X_{ts},X_s \rangle ds$,
we calculate the time evolution of area as
\begin{align*}
	\frac{dA}{dt}&=-\frac{1}{2}\int_0^\L \langle X_t,N\rangle+\langle X,N_t\rangle+\langle X,N\rangle \langle X_{ts},X_s\rangle ds\\
&=-\frac{1}{2}\int_0^\L \langle X_t,N\rangle -\langle X,T\rangle \langle N,X_{ts}\rangle+\langle X,N\rangle \langle X_{ts},X_s\rangle ds\\
&= -\frac{1}{2}\int_0^\L \langle X_t, N\rangle+ \langle X_{ts},\langle X,N\rangle T-\langle X,T\rangle N\rangle ds
\,.
\end{align*}
Now since $\partial_s \left (\langle X,N\rangle T-\langle X,T\rangle N\right )=-N$, integration by parts gives
\begin{align} \label{eq:areaevo}
\frac{dA}{dt}&=-\int_0^\L	\langle X_t,N\rangle ds
\,.
\end{align}

\subsubsection{Estimate for the deficit}
Consider the isoperimetric deficit
\[
	\D:= \L^2-4\pi A\,.
\]
From $\frac{d}{dt} \L =\int \langle kN,\grad \L_X\rangle ds$ and \eqref{eq:areaevo} we find
\[ 
\frac{d}{dt} \D =\int (2\L k-4\pi )\langle N , \grad \L_X \rangle  \, ds
\,.
\]
With the gradient in the form 
\[ \grad \L_X=\int (X(\tilde s)-X(s)) G(s,\tilde s)\, d\tilde s\]
 we use the second order Taylor approximation
\[
	G= \frac{-1}{2\sinh(\L/2)}\left (1+\frac{1}{2} \left (|s-\tilde s|-\L/2 \right )^2+o(\L^4)\right)
\,.
\]
Note that $\int X(\tilde s)-X(s) d \tilde s=\L (\bar X-X(s))$ and moreover
\[\int (2\L k -4\pi )\langle N, \L(X-\bar X)\rangle \,ds = -2\L\D \]
where the $\bar X$ term vanishes because $kN$ and $N$ are both derivatives and $\L \bar X$ is independent of $s$. Hence
\begin{align*}
\frac{d}{dt}\D & =\begin{multlined}[t]
\frac{1}{2\sinh(\L/2)}
\left (-\L\D\left (2+\frac{\L^2}{8}\right ) \right. \\
-\left. \int (2\L k- 4\pi)\langle N,\int (X(\tilde s)-X(s))\left (\frac{1}{2}|s-\tilde s|^2-\L|s-\tilde s|+o(\L^4) \right)\, d\tilde s \, ds \right )
\,.
\end{multlined}\\
\end{align*}
For the terms involving $k$ we have, for example, 
\begin{equation*} 
\begin{multlined}[t]\int \left \langle \L kN, \int (X(\tilde s)-X(s))(s-\tilde s)^2 d\tilde s \right \rangle ds\\
=	-\int \left \langle \L T,\int 2(X(\tilde s)-X(s))(s-\tilde s)-T(s)(s-\tilde s)^2 \, d\tilde s \right \rangle \, ds
\end{multlined}
\end{equation*}
and therefore we estimate
\[
\frac{d}{dt} \D \leq \frac{1}{2\sinh(\L/2)}
\left (-\L\D\left (2+\frac{\L^2}{8}\right ) +o(\L^5) \right) \\
\]
Because $\frac{\L}{2\sinh(\L/2)}\leq 1$ and $\D\ge0$, we have
\[
	\frac{d}{dt} \D \leq -\frac{\L}{2\sinh(\L/2)} 2\D +o(\L^4)
	\,.
\]
For the isoperimetric deficit $\D_Y$ of the asymptotic profile $Y$, we have 
\[ \D_Y=e^{2t}\D,\qquad \frac{d}{dt} \D_Y=e^{2t}\frac{d}{dt} \D +2\D_Y \]
hence 
\[
	\frac{d}{dt} \D_Y \leq 2\D_Y\bigg(1-\frac{\L}{2\sinh(\L/2)}\bigg)+o(\L^4) e^{2t} 
	\,.
\]
From Lemma \ref{alpha} we can take $t\ge t_{3/4}$ such that $o(\L^4) e^{2t}$ decays like $e^{-t}$ for $t>t_{3/4}$.
If $\frac34 \geq 1-\norm{X(0)}_\infty \L(0) $ we can find the required $t_{3/4}$ by solving $\frac34 = 1- \norm{X(0)}_\infty \L(0)e^{-Ct_{3/4}}$, otherwise $t_{3/4}=0$.
The constant $C$ is from the gradient inequality and also depends on $\vn{X(0)}_\infty$.
Therefore the estimate for the integral of the extra terms depends only on $X(0)$.

Integrating with respect to $t$ gives
\[
	\D_Y \leq \D_Y(0)c(X(0))e^{\int_{t_{3/4}}^t1-\frac{\L}{2\sinh(\L/2)} \,d \tau}\,.
\]
The Taylor expansion for $x\mapsto x/(2\sinh(x/2))$ yields
\[
	\D_Y \leq \D_Y(0)c(X(0))e^{\int_{t_{3/4}}^t \frac{\L^2}{24} + o(\L^4) \,d \tau}\,.
\]
Now using again the exponential decay of $\L$ we find
\[
	\D_Y \le c(X(0))\D_X(0)\,.
\]
Summarising, we have:

\begin{proposition}
\label{PNisodef}
Let $X$ be an $H^1(ds)$ curve shortening flow and $Y$ its asymptotic profile.
There is a non-trivial $Y_\infty\in H^1(\S^1,\R^2)$ such that $Y(t)\to Y_\infty$ 
in $C^0$ as $t\to\infty$.
Furthermore, there is a constant $c = c(\vn{X(0)}_\infty)$ such that the isoperimetric deficit of $Y_\infty$ satisfies
\[
	\D_{Y_\infty} \le c\D_{X(0)}
	\,.
\]
\end{proposition}


\subsection{A chord-length estimate and embeddedness}

In this section we prove that if the flow is sufficiently embedded (relative to total length) at any time, it must remain embedded for all future times.
This holds also for the asymptotic profile.

We achieve this via a study of the \emph{squared chord-arc ratio}:
\[ \phi:=\frac{\C^2}{\mcs^2} \]
where, writing $s_1=s(u_1), s_2=s(u_2)$,
\begin{align*} 
\C(t)&:= |X(u_1,t)-X(u_2,t)|\ \  \text{ and}\\
\mcs &:= |s_1-s_2|	
\,.
\end{align*}
Recalling that 
\[ X_t=-X-\int XG\, d\tilde s = -X+\bar X -\int (X-\bar X) G\, d\tilde s \]
we find
\begin{align*}
\frac{d}{dt} \C^2 &= -2\C^2-2\left\langle X(u_1)-X(u_2),\int (X-\bar X) (G(s_1,s)-G(s_2,s))\, d\tilde s\right\rangle 
\end{align*}
and define 
\begin{align*}
	Q_0:=\left\langle X(u_1)-X(u_2),\int (X-\bar X) (G(s_1,s)-G(s_2,s))\, d\tilde s\right\rangle
\end{align*}
so that
\[ \dot \C=-\C- \frac{1}{\C}Q_0. \]
Assuming $s_2>s_1$ we have $\mcs=\int_{s_1}^{s_2} ds$ and from $\frac{d}{dt}ds=\langle X_{ts},X_s \rangle ds$ we obtain 
\[
 \dot  \mcs =\int_{s_1}^{s_2} \langle -T-T\ast G,T\rangle ds=-\int_{s_1}^{s_2}ds-\int_{s_1}^{s_2}\langle T\ast G,T\rangle ds
\,.
\]
Now let 
\[
 Q_1 := \int_{s_1}^{s_2}\langle T\ast G,T\rangle ds
\]
and then 
\[
 \dot S =-\mcs -Q_1 
\,.
\]
Therefore the time evolution of the squared chord-arc ratio is given by
\[
 \dot \phi = \frac{2\C \dot \C}{S^2}-2\frac{C^2\dot \mcs}{S^3}=2\phi \frac{\dot \C}{\C}-2\phi\frac{\dot \mcs}{\mcs}=-2\phi \frac{Q_0}{\C^2}+2\phi \frac{Q_1}{\mcs}
\,.
\]
Using the estimates (that follow via Poincar\'e and \eqref{eq:TGbound})
\begin{align*}
	|Q_0|&\leq \frac{\L^2}{2} \C \mcs \\
|Q_1|&\leq\frac{\L^2}{2} \mcs 
\end{align*} 
and recalling the length decay estimate \eqref{eq:Lexp} we see that
\begin{align*}
\frac{d}{dt}\phi &= -2\phi \frac{Q_0}{\C^2}+2\phi \frac{Q_1}{\mcs}
\\
&\geq -\L^2\sqrt\phi(1+\sqrt\phi)
\,.
\end{align*}
Therefore
\[
\sqrt\phi' \ge -\frac12\L^2\sqrt\phi - \frac12\L^2
\,.
\]
Lemma \ref{LMexpdeclen}, and choosing the appropriate $\varepsilon$ in the proof of Lemma \ref{lem:gradientinequality}, implies that
\[
\L(t) \le \L(0)e^{-\beta(X_0) t}
\]
where $\beta(X_0) = 1/\sqrt{2+\vn{X_0}_\infty^2}$.
	
Now let us impose the following hypothesis on $X_0$:
\begin{equation}
\label{EQhypo}
\inf_{s\in[0,\L_0]} \sqrt{\phi_0}(s) 
> \frac{L_0^2}{4\beta(X_0)} e^{\frac{L_0^2}{4\beta(X_0)}}
\,.
\end{equation}

We calculate
\begin{align*}
\frac{d}{dt}\bigg(
	e^{\frac12\int_0^t \L^2(\tau)\,d\tau}\sqrt\phi
\bigg)
	&\ge
		e^{\frac12\int_0^t \L^2(\tau)\,d\tau}\Big(-\frac12\L^2\Big)
\\
	&\ge
		-\frac12\L_0^2e^{-2\beta(X_0)t + \frac12\int_0^t \L_0^2e^{-2\beta(X_0)\tau}\,d\tau}
\\
	&\ge
		-\frac12\L_0^2e^{-2\beta(X_0)t + \frac{\L_0^2}{4\beta(X_0)}}
	=
		-\frac12\L_0^2e^{\frac{\L_0^2}{4\beta(X_0)}}e^{-2\beta(X_0)t}
\,.
\end{align*}
Integration gives
\[
	e^{\frac12\int_0^t \L^2(\tau)\,d\tau}\sqrt\phi
	\ge \sqrt{\phi_0} - \frac{\L_0^2}{4\beta(X_0)} e^{\frac{\L_0^2}{4\beta(X_0)}}
\,.
\]
By hypothesis \eqref{EQhypo} the RHS is positive, and so the function $\sqrt{\phi}$ can never vanish.
Since the chord-arc length ratio is scale-invariant, the same is true for the asymptotic profile $Y$.
Moreover, the hypothesis \eqref{EQhypo} may be satisfied simply by scaling any
embedded initial data (again, $\phi$ is scale-invariant, but the RHS of
\eqref{EQhypo} is not).
Thus we have the following result:

\begin{proposition}
\label{PNchordarc}
Let $X$ be an $H^1(ds)$ curve shortening flow.
Suppose $X_0\in \imm^1$ satisfies
\begin{equation*}
\inf_{s\in[0,\L_0]} \frac{\C(s)}{\mcs(s)} 
> \frac{L_0^2\sqrt{2+\vn{X_0}_\infty^2}}{4} e^{\frac{L_0^2\sqrt{2+\vn{X_0}_\infty^2}}{4}}
\end{equation*}
where, writing $s_1=s(u_1), s_2=s(u_2)$,
\begin{align*} 
\C(t) := |X(u_1,t)-X(u_2,t)|\ \  \text{ and}\ \ \ \mcs := |s_1-s_2|	
\,.
\end{align*}
Then $X$ (as well as its asymptotic profile and limit $Y_\infty$) is a family of embeddings.
\end{proposition}

Proposition \ref{PNchordarc} together with Theorem \ref{TMglobCk} completes the proof of Theorem \ref{ckcharch} from the introduction. Moreover Theorem \ref{TMcurvatureprofile}, Proposition \ref{PNisodef} and Proposition \ref{PNchordarc} complete the proof of Theorem \ref{besttheorem}.


\bibliographystyle{plain}
\bibliography{gradient_flows}

\end{document}

%% file: mmfigure2.tex
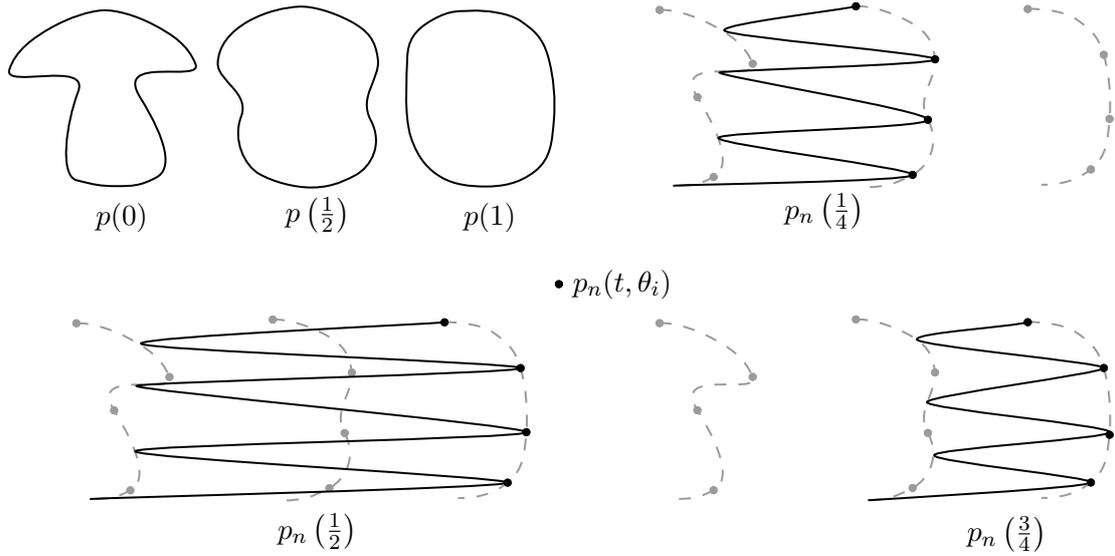
\begin{figure}
\begin{center}

\tikzset{every picture/.style={line width=0.75pt}} 

\tikzset{every picture/.style={line width=0.75pt}} 

\begin{tikzpicture}[x=0.75pt,y=0.75pt,yscale=-0.75,xscale=0.75]

\draw  [line width=0.75]  (70.26,23.53) .. controls (78.8,23.04) and (92.05,26.69) .. (104.29,33.15) .. controls (116.52,39.6) and (127.18,46.53) .. (132.24,59.57) .. controls (137.31,72.62) and (77.45,52.75) .. (95.42,81.98) .. controls (113.38,111.2) and (116.97,129.71) .. (106.19,135.56) .. controls (95.42,141.4) and (90.03,141.4) .. (79.25,141.4) .. controls (68.47,141.4) and (64.88,140.43) .. (53.2,137.5) .. controls (41.52,134.58) and (46.01,99.51) .. (54.1,78.08) .. controls (62.18,56.65) and (4.39,76.3) .. (8.29,63.47) .. controls (12.18,50.63) and (24.01,42.04) .. (36.24,35.09) .. controls (48.48,28.15) and (61.73,24.01) .. (70.26,23.53) -- cycle ;
\draw  [line width=0.75]  (200.95,21.04) .. controls (209.48,20.56) and (222.73,24.21) .. (234.97,30.66) .. controls (247.2,37.12) and (252.15,43.58) .. (253.62,56.62) .. controls (255.09,69.67) and (240.45,81.19) .. (248.83,97.16) .. controls (257.21,113.12) and (249.41,128.24) .. (238.64,134.08) .. controls (227.86,139.93) and (219.75,142.32) .. (208.97,142.32) .. controls (198.19,142.32) and (184.09,138.66) .. (174.52,133.17) .. controls (164.95,127.68) and (156.62,110.67) .. (162.61,94.7) .. controls (168.6,78.73) and (146.74,70.68) .. (147.04,59.08) .. controls (147.35,47.47) and (155.16,37.77) .. (167.4,30.83) .. controls (179.64,23.89) and (192.41,21.53) .. (200.95,21.04) -- cycle ;
\draw  [line width=0.75]  (316.28,23.79) .. controls (324.81,23.3) and (341.03,25.23) .. (348.69,30.72) .. controls (356.34,36.21) and (363.04,40.78) .. (366.87,54.5) .. controls (370.7,68.22) and (371.9,83.98) .. (370.7,97.49) .. controls (369.5,111) and (369.03,122.74) .. (358.26,131.34) .. controls (347.48,139.94) and (336.04,141.66) .. (325.26,141.66) .. controls (314.48,141.66) and (304.29,140.43) .. (295.91,135.52) .. controls (287.52,130.61) and (280.34,123.24) .. (275.55,108.5) .. controls (270.76,93.76) and (272.85,67.29) .. (273.16,55.68) .. controls (273.46,44.08) and (277.68,36.83) .. (289.92,29.89) .. controls (302.16,22.94) and (307.75,24.27) .. (316.28,23.79) -- cycle ;
\draw [color={rgb, 255:red, 155; green, 155; blue, 155 }  ,draw opacity=1 ][line width=0.75]  [dash pattern={on 4.5pt off 4.5pt}]  (441.08,23.27) .. controls (449.61,22.78) and (462.86,26.44) .. (475.1,32.89) .. controls (487.34,39.34) and (497.99,46.27) .. (503.06,59.31) .. controls (508.12,72.36) and (448.26,52.5) .. (466.23,81.72) .. controls (484.19,110.95) and (487.79,129.45) .. (477.01,135.3) .. controls (466.23,141.14) and (460.84,141.14) .. (450.06,141.14) ;
\draw [color={rgb, 255:red, 155; green, 155; blue, 155 }  ,draw opacity=1 ][line width=0.75]  [dash pattern={on 4.5pt off 4.5pt}]  (571.76,20.79) .. controls (580.29,20.3) and (593.54,23.95) .. (605.78,30.41) .. controls (618.02,36.86) and (622.96,43.32) .. (624.43,56.37) .. controls (625.9,69.41) and (611.26,80.93) .. (619.64,96.9) .. controls (628.02,112.87) and (620.23,127.98) .. (609.45,133.83) .. controls (598.67,139.67) and (590.56,142.06) .. (579.78,142.06) ;
\draw [line width=0.75]    (441.08,23.27) ;
\draw [color={rgb, 255:red, 155; green, 155; blue, 155 }  ,draw opacity=1 ][line width=0.75]    (441.08,23.27) ;
\draw [shift={(441.08,23.27)}, rotate = 0] [color={rgb, 255:red, 155; green, 155; blue, 155 }  ,draw opacity=1 ][fill={rgb, 255:red, 155; green, 155; blue, 155 }  ,fill opacity=1 ][line width=0.75]      (0, 0) circle [x radius= 2.01, y radius= 2.01]   ;
\draw [shift={(441.08,23.27)}, rotate = 0] [color={rgb, 255:red, 155; green, 155; blue, 155 }  ,draw opacity=1 ][fill={rgb, 255:red, 155; green, 155; blue, 155 }  ,fill opacity=1 ][line width=0.75]      (0, 0) circle [x radius= 2.01, y radius= 2.01]   ;
\draw [line width=0.75]    (571.76,20.79) ;
\draw [shift={(571.76,20.79)}, rotate = 0] [color={rgb, 255:red, 0; green, 0; blue, 0 }  ][line width=0.75]      (0, 0) circle [x radius= 2.01, y radius= 2.01]   ;
\draw [shift={(571.76,20.79)}, rotate = 0] [color={rgb, 255:red, 0; green, 0; blue, 0 }  ][fill={rgb, 255:red, 0; green, 0; blue, 0 }  ][line width=0.75]      (0, 0) circle [x radius= 2.01, y radius= 2.01]   ;
\draw [color={rgb, 255:red, 155; green, 155; blue, 155 }  ,draw opacity=1 ][line width=0.75]    (503.06,59.31) ;
\draw [shift={(503.06,59.31)}, rotate = 0] [color={rgb, 255:red, 155; green, 155; blue, 155 }  ,draw opacity=1 ][line width=0.75]      (0, 0) circle [x radius= 2.01, y radius= 2.01]   ;
\draw [shift={(503.06,59.31)}, rotate = 0] [color={rgb, 255:red, 155; green, 155; blue, 155 }  ,draw opacity=1 ][fill={rgb, 255:red, 155; green, 155; blue, 155 }  ,fill opacity=1 ][line width=0.75]      (0, 0) circle [x radius= 2.01, y radius= 2.01]   ;
\draw [line width=0.75]    (624.43,56.37) ;
\draw [shift={(624.43,56.37)}, rotate = 0] [color={rgb, 255:red, 0; green, 0; blue, 0 }  ][line width=0.75]      (0, 0) circle [x radius= 2.01, y radius= 2.01]   ;
\draw [shift={(624.43,56.37)}, rotate = 0] [color={rgb, 255:red, 0; green, 0; blue, 0 }  ][fill={rgb, 255:red, 0; green, 0; blue, 0 }  ][line width=0.75]      (0, 0) circle [x radius= 2.01, y radius= 2.01]   ;
\draw [line width=0.75]    (619.64,96.9) ;
\draw [shift={(619.64,96.9)}, rotate = 0] [color={rgb, 255:red, 0; green, 0; blue, 0 }  ][line width=0.75]      (0, 0) circle [x radius= 2.01, y radius= 2.01]   ;
\draw [shift={(619.64,96.9)}, rotate = 0] [color={rgb, 255:red, 0; green, 0; blue, 0 }  ][fill={rgb, 255:red, 0; green, 0; blue, 0 }  ][line width=0.75]      (0, 0) circle [x radius= 2.01, y radius= 2.01]   ;
\draw [color={rgb, 255:red, 155; green, 155; blue, 155 }  ,draw opacity=1 ][line width=0.75]    (466.23,81.72) ;
\draw [shift={(466.23,81.72)}, rotate = 0] [color={rgb, 255:red, 155; green, 155; blue, 155 }  ,draw opacity=1 ][line width=0.75]      (0, 0) circle [x radius= 2.01, y radius= 2.01]   ;
\draw [shift={(466.23,81.72)}, rotate = 0] [color={rgb, 255:red, 155; green, 155; blue, 155 }  ,draw opacity=1 ][fill={rgb, 255:red, 155; green, 155; blue, 155 }  ,fill opacity=1 ][line width=0.75]      (0, 0) circle [x radius= 2.01, y radius= 2.01]   ;
\draw [color={rgb, 255:red, 155; green, 155; blue, 155 }  ,draw opacity=1 ][line width=0.75]    (477.01,135.3) ;
\draw [shift={(477.01,135.3)}, rotate = 0] [color={rgb, 255:red, 155; green, 155; blue, 155 }  ,draw opacity=1 ][line width=0.75]      (0, 0) circle [x radius= 2.01, y radius= 2.01]   ;
\draw [shift={(477.01,135.3)}, rotate = 0] [color={rgb, 255:red, 155; green, 155; blue, 155 }  ,draw opacity=1 ][fill={rgb, 255:red, 155; green, 155; blue, 155 }  ,fill opacity=1 ][line width=0.75]      (0, 0) circle [x radius= 2.01, y radius= 2.01]   ;
\draw [line width=0.75]    (609.45,133.83) ;
\draw [shift={(609.45,133.83)}, rotate = 0] [color={rgb, 255:red, 0; green, 0; blue, 0 }  ][line width=0.75]      (0, 0) circle [x radius= 2.01, y radius= 2.01]   ;
\draw [shift={(609.45,133.83)}, rotate = 0] [color={rgb, 255:red, 0; green, 0; blue, 0 }  ][fill={rgb, 255:red, 0; green, 0; blue, 0 }  ][line width=0.75]      (0, 0) circle [x radius= 2.01, y radius= 2.01]   ;
\draw [line width=0.75]    (571.76,20.79) .. controls (563.51,23.14) and (484.08,32.29) .. (484.08,36.86) .. controls (484.08,41.44) and (624.76,52.42) .. (624.43,56.37) .. controls (624.1,60.32) and (486,63.39) .. (481.21,65.22) .. controls (476.43,67.05) and (620.93,90.83) .. (619.64,96.9) .. controls (618.35,102.96) and (480.26,104.56) .. (480.26,109.13) .. controls (480.26,113.7) and (609.45,129.25) .. (609.45,133.83) .. controls (609.45,138.4) and (460.16,140.23) .. (450.06,141.14) ;
\draw [color={rgb, 255:red, 155; green, 155; blue, 155 }  ,draw opacity=1 ][line width=0.75]  [dash pattern={on 4.5pt off 4.5pt}]  (686.14,22.62) .. controls (694.67,22.13) and (710.89,24.06) .. (718.54,29.55) .. controls (726.2,35.04) and (732.9,39.61) .. (736.73,53.33) .. controls (740.55,67.05) and (741.75,82.81) .. (740.55,96.32) .. controls (739.36,109.83) and (738.89,121.57) .. (728.11,130.17) .. controls (717.34,138.77) and (705.9,140.49) .. (695.12,140.49) ;
\draw [color={rgb, 255:red, 155; green, 155; blue, 155 }  ,draw opacity=1 ][line width=0.75]    (686.14,22.62) ;
\draw [shift={(686.14,22.62)}, rotate = 0] [color={rgb, 255:red, 155; green, 155; blue, 155 }  ,draw opacity=1 ][fill={rgb, 255:red, 155; green, 155; blue, 155 }  ,fill opacity=1 ][line width=0.75]      (0, 0) circle [x radius= 2.01, y radius= 2.01]   ;
\draw [shift={(686.14,22.62)}, rotate = 0] [color={rgb, 255:red, 155; green, 155; blue, 155 }  ,draw opacity=1 ][fill={rgb, 255:red, 155; green, 155; blue, 155 }  ,fill opacity=1 ][line width=0.75]      (0, 0) circle [x radius= 2.01, y radius= 2.01]   ;
\draw [color={rgb, 255:red, 155; green, 155; blue, 155 }  ,draw opacity=1 ][line width=0.75]    (728.11,130.17) ;
\draw [shift={(728.11,130.17)}, rotate = 0] [color={rgb, 255:red, 155; green, 155; blue, 155 }  ,draw opacity=1 ][fill={rgb, 255:red, 155; green, 155; blue, 155 }  ,fill opacity=1 ][line width=0.75]      (0, 0) circle [x radius= 2.01, y radius= 2.01]   ;
\draw [shift={(728.11,130.17)}, rotate = 0] [color={rgb, 255:red, 155; green, 155; blue, 155 }  ,draw opacity=1 ][fill={rgb, 255:red, 155; green, 155; blue, 155 }  ,fill opacity=1 ][line width=0.75]      (0, 0) circle [x radius= 2.01, y radius= 2.01]   ;
\draw [color={rgb, 255:red, 155; green, 155; blue, 155 }  ,draw opacity=1 ][line width=0.75]    (740.55,96.32) ;
\draw [shift={(740.55,96.32)}, rotate = 0] [color={rgb, 255:red, 155; green, 155; blue, 155 }  ,draw opacity=1 ][fill={rgb, 255:red, 155; green, 155; blue, 155 }  ,fill opacity=1 ][line width=0.75]      (0, 0) circle [x radius= 2.01, y radius= 2.01]   ;
\draw [shift={(740.55,96.32)}, rotate = 0] [color={rgb, 255:red, 155; green, 155; blue, 155 }  ,draw opacity=1 ][fill={rgb, 255:red, 155; green, 155; blue, 155 }  ,fill opacity=1 ][line width=0.75]      (0, 0) circle [x radius= 2.01, y radius= 2.01]   ;
\draw [color={rgb, 255:red, 155; green, 155; blue, 155 }  ,draw opacity=1 ][line width=0.75]    (736.73,53.33) ;
\draw [shift={(736.73,53.33)}, rotate = 0] [color={rgb, 255:red, 155; green, 155; blue, 155 }  ,draw opacity=1 ][fill={rgb, 255:red, 155; green, 155; blue, 155 }  ,fill opacity=1 ][line width=0.75]      (0, 0) circle [x radius= 2.01, y radius= 2.01]   ;
\draw [shift={(736.73,53.33)}, rotate = 0] [color={rgb, 255:red, 155; green, 155; blue, 155 }  ,draw opacity=1 ][fill={rgb, 255:red, 155; green, 155; blue, 155 }  ,fill opacity=1 ][line width=0.75]      (0, 0) circle [x radius= 2.01, y radius= 2.01]   ;
\draw [color={rgb, 255:red, 155; green, 155; blue, 155 }  ,draw opacity=1 ][line width=0.75]  [dash pattern={on 4.5pt off 4.5pt}]  (52.91,233.39) .. controls (61.44,232.9) and (74.69,236.55) .. (86.93,243.01) .. controls (99.17,249.46) and (109.82,256.39) .. (114.89,269.43) .. controls (119.95,282.48) and (60.1,262.61) .. (78.06,291.84) .. controls (96.03,321.06) and (99.62,339.57) .. (88.84,345.42) .. controls (78.06,351.26) and (72.67,351.26) .. (61.89,351.26) ;
\draw [color={rgb, 255:red, 155; green, 155; blue, 155 }  ,draw opacity=1 ][line width=0.75]  [dash pattern={on 4.5pt off 4.5pt}]  (183.59,230.91) .. controls (192.12,230.42) and (205.37,234.07) .. (217.61,240.53) .. controls (229.85,246.98) and (234.79,253.44) .. (236.26,266.48) .. controls (237.73,279.53) and (223.09,291.05) .. (231.47,307.02) .. controls (239.85,322.99) and (232.06,338.1) .. (221.28,343.95) .. controls (210.5,349.79) and (202.39,352.18) .. (191.61,352.18) ;
\draw [color={rgb, 255:red, 155; green, 155; blue, 155 }  ,draw opacity=1 ][line width=0.75]    (52.91,233.39) ;
\draw [color={rgb, 255:red, 155; green, 155; blue, 155 }  ,draw opacity=1 ][line width=0.75]    (52.91,233.39) ;
\draw [shift={(52.91,233.39)}, rotate = 0] [color={rgb, 255:red, 155; green, 155; blue, 155 }  ,draw opacity=1 ][fill={rgb, 255:red, 155; green, 155; blue, 155 }  ,fill opacity=1 ][line width=0.75]      (0, 0) circle [x radius= 2.01, y radius= 2.01]   ;
\draw [shift={(52.91,233.39)}, rotate = 0] [color={rgb, 255:red, 155; green, 155; blue, 155 }  ,draw opacity=1 ][fill={rgb, 255:red, 155; green, 155; blue, 155 }  ,fill opacity=1 ][line width=0.75]      (0, 0) circle [x radius= 2.01, y radius= 2.01]   ;
\draw [color={rgb, 255:red, 155; green, 155; blue, 155 }  ,draw opacity=1 ][line width=0.75]    (183.59,230.91) ;
\draw [shift={(183.59,230.91)}, rotate = 0] [color={rgb, 255:red, 155; green, 155; blue, 155 }  ,draw opacity=1 ][line width=0.75]      (0, 0) circle [x radius= 2.01, y radius= 2.01]   ;
\draw [shift={(183.59,230.91)}, rotate = 0] [color={rgb, 255:red, 155; green, 155; blue, 155 }  ,draw opacity=1 ][fill={rgb, 255:red, 155; green, 155; blue, 155 }  ,fill opacity=1 ][line width=0.75]      (0, 0) circle [x radius= 2.01, y radius= 2.01]   ;
\draw [color={rgb, 255:red, 155; green, 155; blue, 155 }  ,draw opacity=1 ][line width=0.75]    (114.89,269.43) ;
\draw [shift={(114.89,269.43)}, rotate = 0] [color={rgb, 255:red, 155; green, 155; blue, 155 }  ,draw opacity=1 ][line width=0.75]      (0, 0) circle [x radius= 2.01, y radius= 2.01]   ;
\draw [shift={(114.89,269.43)}, rotate = 0] [color={rgb, 255:red, 155; green, 155; blue, 155 }  ,draw opacity=1 ][fill={rgb, 255:red, 155; green, 155; blue, 155 }  ,fill opacity=1 ][line width=0.75]      (0, 0) circle [x radius= 2.01, y radius= 2.01]   ;
\draw [color={rgb, 255:red, 155; green, 155; blue, 155 }  ,draw opacity=1 ][line width=0.75]    (236.26,266.48) ;
\draw [shift={(236.26,266.48)}, rotate = 0] [color={rgb, 255:red, 155; green, 155; blue, 155 }  ,draw opacity=1 ][line width=0.75]      (0, 0) circle [x radius= 2.01, y radius= 2.01]   ;
\draw [shift={(236.26,266.48)}, rotate = 0] [color={rgb, 255:red, 155; green, 155; blue, 155 }  ,draw opacity=1 ][fill={rgb, 255:red, 155; green, 155; blue, 155 }  ,fill opacity=1 ][line width=0.75]      (0, 0) circle [x radius= 2.01, y radius= 2.01]   ;
\draw [color={rgb, 255:red, 155; green, 155; blue, 155 }  ,draw opacity=1 ][line width=0.75]    (231.47,307.02) ;
\draw [shift={(231.47,307.02)}, rotate = 0] [color={rgb, 255:red, 155; green, 155; blue, 155 }  ,draw opacity=1 ][line width=0.75]      (0, 0) circle [x radius= 2.01, y radius= 2.01]   ;
\draw [shift={(231.47,307.02)}, rotate = 0] [color={rgb, 255:red, 155; green, 155; blue, 155 }  ,draw opacity=1 ][fill={rgb, 255:red, 155; green, 155; blue, 155 }  ,fill opacity=1 ][line width=0.75]      (0, 0) circle [x radius= 2.01, y radius= 2.01]   ;
\draw [color={rgb, 255:red, 155; green, 155; blue, 155 }  ,draw opacity=1 ][line width=0.75]    (78.06,291.84) ;
\draw [shift={(78.06,291.84)}, rotate = 0] [color={rgb, 255:red, 155; green, 155; blue, 155 }  ,draw opacity=1 ][line width=0.75]      (0, 0) circle [x radius= 2.01, y radius= 2.01]   ;
\draw [shift={(78.06,291.84)}, rotate = 0] [color={rgb, 255:red, 155; green, 155; blue, 155 }  ,draw opacity=1 ][fill={rgb, 255:red, 155; green, 155; blue, 155 }  ,fill opacity=1 ][line width=0.75]      (0, 0) circle [x radius= 2.01, y radius= 2.01]   ;
\draw [color={rgb, 255:red, 155; green, 155; blue, 155 }  ,draw opacity=1 ][line width=0.75]    (88.84,345.42) ;
\draw [shift={(88.84,345.42)}, rotate = 0] [color={rgb, 255:red, 155; green, 155; blue, 155 }  ,draw opacity=1 ][line width=0.75]      (0, 0) circle [x radius= 2.01, y radius= 2.01]   ;
\draw [shift={(88.84,345.42)}, rotate = 0] [color={rgb, 255:red, 155; green, 155; blue, 155 }  ,draw opacity=1 ][fill={rgb, 255:red, 155; green, 155; blue, 155 }  ,fill opacity=1 ][line width=0.75]      (0, 0) circle [x radius= 2.01, y radius= 2.01]   ;
\draw [color={rgb, 255:red, 155; green, 155; blue, 155 }  ,draw opacity=1 ][line width=0.75]    (221.28,343.95) ;
\draw [shift={(221.28,343.95)}, rotate = 0] [color={rgb, 255:red, 155; green, 155; blue, 155 }  ,draw opacity=1 ][line width=0.75]      (0, 0) circle [x radius= 2.01, y radius= 2.01]   ;
\draw [shift={(221.28,343.95)}, rotate = 0] [color={rgb, 255:red, 155; green, 155; blue, 155 }  ,draw opacity=1 ][fill={rgb, 255:red, 155; green, 155; blue, 155 }  ,fill opacity=1 ][line width=0.75]      (0, 0) circle [x radius= 2.01, y radius= 2.01]   ;
\draw [color={rgb, 255:red, 0; green, 0; blue, 0 }  ,draw opacity=1 ][line width=0.75]    (297.97,232.74) .. controls (283.48,234.18) and (95.92,242.41) .. (95.92,246.98) .. controls (95.92,251.56) and (348.56,259.79) .. (348.56,263.45) .. controls (348.56,267.11) and (98.79,272.6) .. (93.05,275.34) .. controls (87.3,278.08) and (352.39,301.87) .. (352.39,306.44) .. controls (352.39,311.01) and (92.09,314.67) .. (92.09,319.25) .. controls (92.09,323.82) and (341.86,336.63) .. (339.95,340.29) .. controls (338.03,343.95) and (71.04,350.35) .. (61.89,351.26) ;
\draw [color={rgb, 255:red, 155; green, 155; blue, 155 }  ,draw opacity=1 ][line width=0.75]  [dash pattern={on 4.5pt off 4.5pt}]  (297.97,232.74) .. controls (306.5,232.25) and (322.72,234.18) .. (330.38,239.66) .. controls (338.03,245.15) and (344.73,249.73) .. (348.56,263.45) .. controls (352.39,277.17) and (353.58,292.93) .. (352.39,306.44) .. controls (351.19,319.95) and (350.72,331.69) .. (339.95,340.29) .. controls (329.17,348.88) and (317.73,350.61) .. (306.95,350.61) ;
\draw [line width=0.75]    (297.97,232.74) ;
\draw [shift={(297.97,232.74)}, rotate = 0] [color={rgb, 255:red, 0; green, 0; blue, 0 }  ][fill={rgb, 255:red, 0; green, 0; blue, 0 }  ][line width=0.75]      (0, 0) circle [x radius= 2.01, y radius= 2.01]   ;
\draw [shift={(297.97,232.74)}, rotate = 0] [color={rgb, 255:red, 0; green, 0; blue, 0 }  ][fill={rgb, 255:red, 0; green, 0; blue, 0 }  ][line width=0.75]      (0, 0) circle [x radius= 2.01, y radius= 2.01]   ;
\draw [line width=0.75]    (339.95,340.29) ;
\draw [shift={(339.95,340.29)}, rotate = 0] [color={rgb, 255:red, 0; green, 0; blue, 0 }  ][fill={rgb, 255:red, 0; green, 0; blue, 0 }  ][line width=0.75]      (0, 0) circle [x radius= 2.01, y radius= 2.01]   ;
\draw [shift={(339.95,340.29)}, rotate = 0] [color={rgb, 255:red, 0; green, 0; blue, 0 }  ][fill={rgb, 255:red, 0; green, 0; blue, 0 }  ][line width=0.75]      (0, 0) circle [x radius= 2.01, y radius= 2.01]   ;
\draw [line width=0.75]    (352.39,306.44) ;
\draw [shift={(352.39,306.44)}, rotate = 0] [color={rgb, 255:red, 0; green, 0; blue, 0 }  ][fill={rgb, 255:red, 0; green, 0; blue, 0 }  ][line width=0.75]      (0, 0) circle [x radius= 2.01, y radius= 2.01]   ;
\draw [shift={(352.39,306.44)}, rotate = 0] [color={rgb, 255:red, 0; green, 0; blue, 0 }  ][fill={rgb, 255:red, 0; green, 0; blue, 0 }  ][line width=0.75]      (0, 0) circle [x radius= 2.01, y radius= 2.01]   ;
\draw [line width=0.75]    (348.56,263.45) ;
\draw [shift={(348.56,263.45)}, rotate = 0] [color={rgb, 255:red, 0; green, 0; blue, 0 }  ][fill={rgb, 255:red, 0; green, 0; blue, 0 }  ][line width=0.75]      (0, 0) circle [x radius= 2.01, y radius= 2.01]   ;
\draw [shift={(348.56,263.45)}, rotate = 0] [color={rgb, 255:red, 0; green, 0; blue, 0 }  ][fill={rgb, 255:red, 0; green, 0; blue, 0 }  ][line width=0.75]      (0, 0) circle [x radius= 2.01, y radius= 2.01]   ;
\draw [color={rgb, 255:red, 155; green, 155; blue, 155 }  ,draw opacity=1 ][line width=0.75]  [dash pattern={on 4.5pt off 4.5pt}]  (441.08,233.3) .. controls (449.61,232.82) and (462.86,236.47) .. (475.1,242.92) .. controls (487.34,249.38) and (497.99,256.3) .. (503.06,269.35) .. controls (508.12,282.39) and (448.26,262.53) .. (466.23,291.75) .. controls (484.19,320.98) and (487.79,339.49) .. (477.01,345.33) .. controls (466.23,351.18) and (460.84,351.18) .. (450.06,351.18) ;
\draw [color={rgb, 255:red, 155; green, 155; blue, 155 }  ,draw opacity=1 ][line width=0.75]  [dash pattern={on 4.5pt off 4.5pt}]  (571.76,230.82) .. controls (580.29,230.33) and (593.54,233.99) .. (605.78,240.44) .. controls (618.02,246.89) and (622.96,253.35) .. (624.43,266.4) .. controls (625.9,279.44) and (611.26,290.96) .. (619.64,306.93) .. controls (628.02,322.9) and (620.23,338.01) .. (609.45,343.86) .. controls (598.67,349.7) and (590.56,352.09) .. (579.78,352.09) ;
\draw [color={rgb, 255:red, 155; green, 155; blue, 155 }  ,draw opacity=1 ][line width=0.75]    (441.08,233.3) ;
\draw [color={rgb, 255:red, 155; green, 155; blue, 155 }  ,draw opacity=1 ][line width=0.75]    (441.08,233.3) ;
\draw [shift={(441.08,233.3)}, rotate = 0] [color={rgb, 255:red, 155; green, 155; blue, 155 }  ,draw opacity=1 ][fill={rgb, 255:red, 155; green, 155; blue, 155 }  ,fill opacity=1 ][line width=0.75]      (0, 0) circle [x radius= 2.01, y radius= 2.01]   ;
\draw [shift={(441.08,233.3)}, rotate = 0] [color={rgb, 255:red, 155; green, 155; blue, 155 }  ,draw opacity=1 ][fill={rgb, 255:red, 155; green, 155; blue, 155 }  ,fill opacity=1 ][line width=0.75]      (0, 0) circle [x radius= 2.01, y radius= 2.01]   ;
\draw [color={rgb, 255:red, 155; green, 155; blue, 155 }  ,draw opacity=1 ][line width=0.75]    (571.76,230.82) ;
\draw [shift={(571.76,230.82)}, rotate = 0] [color={rgb, 255:red, 155; green, 155; blue, 155 }  ,draw opacity=1 ][line width=0.75]      (0, 0) circle [x radius= 2.01, y radius= 2.01]   ;
\draw [shift={(571.76,230.82)}, rotate = 0] [color={rgb, 255:red, 155; green, 155; blue, 155 }  ,draw opacity=1 ][fill={rgb, 255:red, 155; green, 155; blue, 155 }  ,fill opacity=1 ][line width=0.75]      (0, 0) circle [x radius= 2.01, y radius= 2.01]   ;
\draw [color={rgb, 255:red, 155; green, 155; blue, 155 }  ,draw opacity=1 ][line width=0.75]    (503.06,269.35) ;
\draw [shift={(503.06,269.35)}, rotate = 0] [color={rgb, 255:red, 155; green, 155; blue, 155 }  ,draw opacity=1 ][line width=0.75]      (0, 0) circle [x radius= 2.01, y radius= 2.01]   ;
\draw [shift={(503.06,269.35)}, rotate = 0] [color={rgb, 255:red, 155; green, 155; blue, 155 }  ,draw opacity=1 ][fill={rgb, 255:red, 155; green, 155; blue, 155 }  ,fill opacity=1 ][line width=0.75]      (0, 0) circle [x radius= 2.01, y radius= 2.01]   ;
\draw [color={rgb, 255:red, 155; green, 155; blue, 155 }  ,draw opacity=1 ][line width=0.75]    (624.43,266.4) ;
\draw [shift={(624.43,266.4)}, rotate = 0] [color={rgb, 255:red, 155; green, 155; blue, 155 }  ,draw opacity=1 ][line width=0.75]      (0, 0) circle [x radius= 2.01, y radius= 2.01]   ;
\draw [shift={(624.43,266.4)}, rotate = 0] [color={rgb, 255:red, 155; green, 155; blue, 155 }  ,draw opacity=1 ][fill={rgb, 255:red, 155; green, 155; blue, 155 }  ,fill opacity=1 ][line width=0.75]      (0, 0) circle [x radius= 2.01, y radius= 2.01]   ;
\draw [color={rgb, 255:red, 155; green, 155; blue, 155 }  ,draw opacity=1 ][line width=0.75]    (619.64,306.93) ;
\draw [shift={(619.64,306.93)}, rotate = 0] [color={rgb, 255:red, 155; green, 155; blue, 155 }  ,draw opacity=1 ][line width=0.75]      (0, 0) circle [x radius= 2.01, y radius= 2.01]   ;
\draw [shift={(619.64,306.93)}, rotate = 0] [color={rgb, 255:red, 155; green, 155; blue, 155 }  ,draw opacity=1 ][fill={rgb, 255:red, 155; green, 155; blue, 155 }  ,fill opacity=1 ][line width=0.75]      (0, 0) circle [x radius= 2.01, y radius= 2.01]   ;
\draw [color={rgb, 255:red, 155; green, 155; blue, 155 }  ,draw opacity=1 ][line width=0.75]    (466.23,291.75) ;
\draw [shift={(466.23,291.75)}, rotate = 0] [color={rgb, 255:red, 155; green, 155; blue, 155 }  ,draw opacity=1 ][line width=0.75]      (0, 0) circle [x radius= 2.01, y radius= 2.01]   ;
\draw [shift={(466.23,291.75)}, rotate = 0] [color={rgb, 255:red, 155; green, 155; blue, 155 }  ,draw opacity=1 ][fill={rgb, 255:red, 155; green, 155; blue, 155 }  ,fill opacity=1 ][line width=0.75]      (0, 0) circle [x radius= 2.01, y radius= 2.01]   ;
\draw [color={rgb, 255:red, 155; green, 155; blue, 155 }  ,draw opacity=1 ][line width=0.75]    (477.01,345.33) ;
\draw [shift={(477.01,345.33)}, rotate = 0] [color={rgb, 255:red, 155; green, 155; blue, 155 }  ,draw opacity=1 ][line width=0.75]      (0, 0) circle [x radius= 2.01, y radius= 2.01]   ;
\draw [shift={(477.01,345.33)}, rotate = 0] [color={rgb, 255:red, 155; green, 155; blue, 155 }  ,draw opacity=1 ][fill={rgb, 255:red, 155; green, 155; blue, 155 }  ,fill opacity=1 ][line width=0.75]      (0, 0) circle [x radius= 2.01, y radius= 2.01]   ;
\draw [color={rgb, 255:red, 155; green, 155; blue, 155 }  ,draw opacity=1 ][line width=0.75]    (609.45,343.86) ;
\draw [shift={(609.45,343.86)}, rotate = 0] [color={rgb, 255:red, 155; green, 155; blue, 155 }  ,draw opacity=1 ][line width=0.75]      (0, 0) circle [x radius= 2.01, y radius= 2.01]   ;
\draw [shift={(609.45,343.86)}, rotate = 0] [color={rgb, 255:red, 155; green, 155; blue, 155 }  ,draw opacity=1 ][fill={rgb, 255:red, 155; green, 155; blue, 155 }  ,fill opacity=1 ][line width=0.75]      (0, 0) circle [x radius= 2.01, y radius= 2.01]   ;
\draw [color={rgb, 255:red, 0; green, 0; blue, 0 }  ,draw opacity=1 ][line width=0.75]    (686.14,232.65) .. controls (674.52,235.92) and (612.32,239.58) .. (612.32,244.15) .. controls (612.32,248.73) and (736.73,259.7) .. (736.73,263.36) .. controls (736.73,267.02) and (619.98,281.66) .. (619.02,286.23) .. controls (618.06,290.81) and (740.55,301.78) .. (740.55,306.36) .. controls (740.55,310.93) and (623.8,317.33) .. (623.8,321.91) .. controls (623.8,326.48) and (730.03,336.54) .. (728.11,340.2) .. controls (726.2,343.86) and (591.27,351.18) .. (579.78,352.09) ;
\draw [color={rgb, 255:red, 155; green, 155; blue, 155 }  ,draw opacity=1 ][line width=0.75]  [dash pattern={on 4.5pt off 4.5pt}]  (686.14,232.65) .. controls (694.67,232.16) and (710.89,234.09) .. (718.54,239.58) .. controls (726.2,245.07) and (732.9,249.64) .. (736.73,263.36) .. controls (740.55,277.08) and (741.75,292.84) .. (740.55,306.36) .. controls (739.36,319.87) and (738.89,331.6) .. (728.11,340.2) .. controls (717.34,348.8) and (705.9,350.52) .. (695.12,350.52) ;
\draw [line width=0.75]    (686.14,232.65) ;
\draw [shift={(686.14,232.65)}, rotate = 0] [color={rgb, 255:red, 0; green, 0; blue, 0 }  ][fill={rgb, 255:red, 0; green, 0; blue, 0 }  ][line width=0.75]      (0, 0) circle [x radius= 2.01, y radius= 2.01]   ;
\draw [shift={(686.14,232.65)}, rotate = 0] [color={rgb, 255:red, 0; green, 0; blue, 0 }  ][fill={rgb, 255:red, 0; green, 0; blue, 0 }  ][line width=0.75]      (0, 0) circle [x radius= 2.01, y radius= 2.01]   ;
\draw [line width=0.75]    (728.11,340.2) ;
\draw [shift={(728.11,340.2)}, rotate = 0] [color={rgb, 255:red, 0; green, 0; blue, 0 }  ][fill={rgb, 255:red, 0; green, 0; blue, 0 }  ][line width=0.75]      (0, 0) circle [x radius= 2.01, y radius= 2.01]   ;
\draw [shift={(728.11,340.2)}, rotate = 0] [color={rgb, 255:red, 0; green, 0; blue, 0 }  ][fill={rgb, 255:red, 0; green, 0; blue, 0 }  ][line width=0.75]      (0, 0) circle [x radius= 2.01, y radius= 2.01]   ;
\draw [line width=0.75]    (740.55,308.19) ;
\draw [shift={(740.55,308.19)}, rotate = 0] [color={rgb, 255:red, 0; green, 0; blue, 0 }  ][fill={rgb, 255:red, 0; green, 0; blue, 0 }  ][line width=0.75]      (0, 0) circle [x radius= 2.01, y radius= 2.01]   ;
\draw [shift={(740.55,308.19)}, rotate = 0] [color={rgb, 255:red, 0; green, 0; blue, 0 }  ][fill={rgb, 255:red, 0; green, 0; blue, 0 }  ][line width=0.75]      (0, 0) circle [x radius= 2.01, y radius= 2.01]   ;
\draw [line width=0.75]    (736.73,263.36) ;
\draw [shift={(736.73,263.36)}, rotate = 0] [color={rgb, 255:red, 0; green, 0; blue, 0 }  ][fill={rgb, 255:red, 0; green, 0; blue, 0 }  ][line width=0.75]      (0, 0) circle [x radius= 2.01, y radius= 2.01]   ;
\draw [shift={(736.73,263.36)}, rotate = 0] [color={rgb, 255:red, 0; green, 0; blue, 0 }  ][fill={rgb, 255:red, 0; green, 0; blue, 0 }  ][line width=0.75]      (0, 0) circle [x radius= 2.01, y radius= 2.01]   ;
\draw [line width=0.75]    (374.08,207.14) ;
\draw [shift={(374.08,207.14)}, rotate = 0] [color={rgb, 255:red, 0; green, 0; blue, 0 }  ][fill={rgb, 255:red, 0; green, 0; blue, 0 }  ][line width=0.75]      (0, 0) circle [x radius= 2.01, y radius= 2.01]   ;
\draw [shift={(374.08,207.14)}, rotate = 0] [color={rgb, 255:red, 0; green, 0; blue, 0 }  ][fill={rgb, 255:red, 0; green, 0; blue, 0 }  ][line width=0.75]      (0, 0) circle [x radius= 2.01, y radius= 2.01]   ;

\draw (63.82,150.78) node [anchor=north west][inner sep=0.75pt]    {$p( 0)$};
\draw (190.1,146.57) node [anchor=north west][inner sep=0.75pt]    {$p\left(\tfrac{1}{2}\right)$};
\draw (308.81,150.78) node [anchor=north west][inner sep=0.75pt]    {$p( 1)$};
\draw (522.37,144.23) node [anchor=north west][inner sep=0.75pt]    {$p_{n}\left(\tfrac{1}{4}\right)$};
\draw (184.93,358.92) node [anchor=north west][inner sep=0.75pt]    {$p_{n}\left(\tfrac{1}{2}\right)$};
\draw (643.91,360.66) node [anchor=north west][inner sep=0.75pt]    {$p_{n}\left(\tfrac{3}{4}\right)$};
\draw (381.36,195.27) node [anchor=north west][inner sep=0.75pt]    {$p_{n}( t,\theta _{i})$};

\end{tikzpicture}

\end{center}

\label{FG1}
\caption{The distortion of a path $p$ to one that is shorter, according to the $L^2(ds)$-metric. For the sake of clarity, only one side of the curves on the shorter path $p_n$ is pictured.}
\end{figure}